\RequirePackage{fix-cm}

\documentclass[smallextended,envcountsect]{svjour3}       

\smartqed  

\usepackage{cite}
\usepackage{dsfont}
\usepackage{graphicx}
\usepackage{amssymb,amsfonts,amsmath}
\usepackage{xcolor}
\usepackage{enumitem}
\usepackage{eurosym}

\newtheorem{assumptions}{Assumptions}[section]

\newcommand{\dom}[1]{\mathrm{dom}(#1)}
\newcommand{\epi}{\mathop{\rm epi}}

\renewcommand{\phi}{\varphi}

\newcommand{\eps}{\varepsilon}

\newcommand{\norm}[1]{\left\Vert#1\right\Vert}

\newcommand{\scr}[1]{\mathcal{#1}} 

\newcommand{\dl}[1]{\mathbb{#1}}
\renewcommand{\Re}{\dl{R}}

\newcommand{\col}[1]{\left\{#1\right\}}

\newcommand{\conv}{\mathop{\rm co}}

\newcommand{\inte}{\mathop{\rm int}}

\newcommand{\bd}{\mathop{\rm bd}}
\newcommand{\cl}[1]{{\rm cl}^{#1}}

\newcommand{\gph}{\mathop{\rm gph}}

\newcommand{\pmord}{\partial}
\newcommand{\pfrech}{\hat{\partial}}

\newcommand{\nge}{\mathop{\rm N}\nolimits}
\newcommand{\nfrech}{\hat{\nge}}
\newcommand{\nmord}{\nge}

\newcommand{\fc}[2]{: #1 \rightarrow #2}

\newcommand{\prb}{\dl{P}}

\newcommand{\sph}{\dl{S}^{m-1}}

\let\epsilon\varepsilon

\usepackage{todonotes}

\usepackage{soul}

\journalname{Working Paper}

\newcommand{\tto}{\rightrightarrows}
\def\dst{\operatorname{d}}
\DeclareMathOperator*{\esssup}{ess\,sup}
\renewcommand{\H}{\mathcal{X}}
\renewcommand{\S}{\mathcal{S}}
\newcommand{\G}{\mathcal{G}}
\newcommand{\zep}[3]{{z_{#1}^{#2,#3}}}
\begin{document}
	
	
	\title{Probability functions generated by set-valued mappings: a study of first order information
	} \subtitle{}
	
	
	\author{Wim van Ackooij \and  Pedro {P{\'e}rez-Aros}\and Claudia Soto }
	
	
	\institute{   W. van Ackooij \at
		EDF R\&D. OSIRIS
		7, Boulevard Gaspard Monge, F-91120 Palaiseau, France\\
		\email{wim.van-ackooij@edf.fr}            \\
		\and
		P. {P{\'e}rez-Aros} \at
		Instituto de Ciencias de la Ingenier\'ia, Universidad de O'Higgins, Rancagua, Chile\\
		\email{pedro.perez@uoh.cl}\\
			\and
		C. {Soto} \at
		Universidad de Chile, DIM-CMM and Instituto de Ciencias de la Ingenier\'ia, Universidad de O'Higgins, Rancagua, Chile\\
		\email{csoto@dim.uchile.cl}
	}
	
	\date{Received: date / Accepted: date}

	\maketitle
	
	\begin{abstract}
		Probability functions appear in constraints of many optimization problems in practice and have become quite popular. Understanding their first-order properties has proven useful, not only theoretically but also in implementable algorithms, giving rise to competitive algorithms in several situations. Probability functions are built up from a random vector belonging to some parameter-dependent subset of the range of a given random vector. In this paper, we investigate first order information of probability functions specified through a convex-valued set-valued application. We provide conditions under which the resulting probability function is indeed locally Lipschitzian as well as subgradient formul{\ae}. The resulting formul{\ae} are made concrete in a classic optimization setting and put to work in an illustrative example coming from an energy application.
		
	\end{abstract}
	
	\section{Introduction}
	
	In many problems from practice, one faces the situation wherein a decision $x \in \Re^n$ has to be taken, ensuring some best economic outcome, all while ensuring the satisfaction of several constraints. The latter constraints usually represent to some degree a set of physical restrictions or rules of use. It is not infrequent that the outcome of the ``constraints" is equally impacted by the random realization of a (random) vector. Thus, it becomes necessary to give an appropriate meaning to what it means for $x$ to satisfy the ``constraints." An intuitive and elegant way of doing so is by making use of so-called chance or probabilistic/probability constraints. The idea is clear: feasibility of $x$ means that with high enough probability the given inequality constraints hold true. This can thus be formalized as follows. We are given a mapping $g \fc{\Re^n \times \Re^m}{\Re^k}$, a random vector $\xi \in \Re^m$, typically having a density w.r.t. the Lebesgue measure and we thus define the probability function $\phi \fc{\Re^n}{[0,1]}$ by setting:
	\begin{equation*}
	    \phi(x) := \prb\left( g(x,\xi) \leq 0 \right). 
	\end{equation*}
	The user (or decision maker) then stipulates a safety-level $p \in (0,1)$ and feasibility is understood as $x$ satisfying $\phi(x) \geq p$ on top of any other ``simple" or ``deterministic" constraints. As we have just specified, the mapping $\phi$ itself is clearly a non-linear map and can itself not be concave, unless it is constant. This does not mean however that the probabilistic constraint itself can not define a convex feasible set - that is indeed very well possible! The interrogation of possible analytic properties of $\phi$, its upper-level sets, or optimization problems having $\phi$ as a constituent is natural. We refer the reader to 
	\cite{vanAckooij_2020} for a recent overview of the many possibilities. For instance, statements regarding convexity of upper-level sets of probability functions are available. Results range from Prékopa's classic ``log-concavity" Theorem, e.g., \cite{Prekopa_1971} asserting convexity of all upper-level sets to recent investigations regarding ``eventual convexity", e.g., \cite{Henrion_Strugarek_2008,vanAckooij_Laguel_Malick_Matiussi-Ramalho_2022} - ensuring convexity of the upper-level sets beyond a given computable level.
	
	Be this as it may, significant effort has also been put into understanding ``first-order" properties of $\phi$. This is quite natural if one thinks of setting up computational approaches for actually solving the underlying optimization problems. On this front too, classic results are available regarding the differentiability of the probability function $\phi$, e.g., \cite{Uryasev_1995}. The latter results are quite general - and somewhat abstract - in nature. To the best of our knowledge, we are unaware of a concrete implementation of any ``gradient formul\ae\/" from that string of research. Motivated by being able to efficiently compute or estimate a ``gradient" of the probability function, a different approach has turned out to be quite fruitful. Originally, by assuming $\xi$ to be of a specific parametric type and assuming $g$ to have additional structure, in addition to expected properties such as continuous differentiability, progress in this direction has been made, e.g., \cite{Royset_Polak_2004,Royset_Polak_2007,vanAckooij_Henrion_2014}. It was also found that by using tools from nonsmooth analysis, more could be stated, and a great many assumptions relaxed. We can briefly cite, 
	\cite{vanAckooij_Henrion_2016, vanAckooij_Malick_2016, vanAckooij_Henrion_Perez-Aros_2019, vanAckooij_Javal_Perez-Aros_2020, vanAckooij_Perez-Aros_2019, vanAckooij_Perez-Aros_2020, vanAckooij_Perez-Aros_2021, vanAckooij_Perez-Aros_2022, Hantoute_Henrion_Perez-Aros_2017} for some contributions to this string of research. The latter results leverage, in part, on recent investigations concerning Leibniz-like rules for nonsmooth analysis, e.g., \cite{Correa_Hantoute_Perez-Aros_2019, Correa_Hantoute_Perez-Aros_2020}.
	
	This, of course, interrogates on the reason for these extensions and in particular how they relate to practice. In fact, it is inspiration from practice that drives the development. This will also be true of the current work. In fact, we will be interested in a situation wherein for a fixed $x$, the mapping $g$ itself is not convex in the second argument, but does define a convex set. This is a relevant case if for instance some components of the mapping $g$ are specified through fractions - with convex numerators and concave denominators, whereas other components may specify a convex inequality simply. The application from energy that we will consider falls in this class. It turns out that a fruitful way of looking at the question takes abstraction from the idea of an underlying system of inequalities and directly considers a set-valued application. Formally, the main object studied in this work corresponds to the following probability function 
	\begin{align}\label{P00} 
		\varphi(x):= \mathbb{P} (\omega \in \Omega: \xi(\omega) \in \S_i(x) \text{ for all }i=1,\ldots,s),
	\end{align}
	where  $\xi: \Omega \to \Re^m$ is a random vector from a probability space $(\Omega,\mathcal{A},\mathbb{P})$ and $\S_i:\H \tto \Re^m$ with $i=1,\dots,s$ is a family of set-valued mappings defined in some separable reflexive Banach space $\H$. The abstract set-valued formulation \eqref{P00} of a  probability function will allow us to understand the variational properties of the function $\varphi$ (Lipschitz continuity and subdifferentials) in terms of the  variational properties of the set-valued mappings $\S_i$ (Lipschitz-like continuity  and coderivatives). Here, it is important to note that this approach takes advantage of the fact that the Lipschitz-like continuity and coderivatives are proper to the set $\S_i$ and do not depend on some specific representation given by sublevel sets $g(x,z) \leq 0$, as in classical studies. Consequently the formul{\ae} derived here provide a more general take and allow us to consider situations not covered previously.
	
	This paper is organised as follows. Section \ref{section:preliminaries} introduces the setting and notation, and recalls some important notions used throughout this paper. Section \ref{section:inner_enlargement} formulates and examines an inner enlargement of the probability function \eqref{P00}. After establishing some preliminary results, a subdifferential formula and the Lipschitz continuity of the probability function \eqref{P00} are provided in Section \ref{section:main}. 
	Finally, Section \ref{sec:applications} provides the description of a case from energy where our results are illustrated in a computational experiment.

	\section{Notation, background and blanket setting}
	\label{section:preliminaries}
	
	\subsection{Basic notation}
	In the following, $(\H$,$\|\cdot\|)$ will be a separable reflexive Banach space, $\H^\ast$ its topological dual and duality product $\langle x^\ast,x\rangle = x^*(x)$ for $x\in\H$, $x^\ast\in \H^\ast$. For a point $x\in\H$ and $r\geq 0$ the closed ball of radius $r$ and centered at $x$ is denoted by $\mathbb{B}_r(x)$. Similarly for $x^\ast\in\H^\ast$, $\mathbb{B}_r(x^\ast)$ denotes the closed ball in the dual space. The unit balls in $\H$ and $\H^\ast$ are denoted by $\mathbb{B}$ and $\mathbb{B}^\ast$, respectively. The norm on $\H^\ast$ is denoted by $\|\cdot\|_\ast$ and let us denote by $\cl{w^\ast}C$ the closure of a subset $C\subset \H^\ast$ with respect to $w^\ast$, the weak$^\ast$-topology on $\H^\ast$. We adopt the following notation: $\rightarrow$ (respectively $\rightharpoonup$) denotes the convergence with respect to the norm-topology (respectively the weak$^\ast$-topology). Furthermore, for a function $f$ and subset $C\subseteq \H$ we write $x_k\xrightarrow{f} x$ (respectively $x_k\xrightarrow{C} x$) to mean that $x_k\to x$ with $f(x_k)\to f(x)$ (respectively $x_k\to x$ with $x_k\in C$).
	
	We also use $\|\cdot\|$ to denote the Euclidean norm of the Euclidean space $\mathbb{R}^m$ and for a given nonempty set $C\subseteq \Re^m$ we let $\bd C$ and $\inte C$ denote the boundary and interior of $C$, respectively. The set of points in $C$ at minimum distance from $z\in\Re^m$ is the projection $P_C(z)$ and that minimum distance is defined by $\dst(z,C):=\inf\{\|z-w\|:w\in C\}$. We recall that when a nonempty set $C\subseteq \Re^m$ is furthermore closed and convex, $P_C(z)$ is the unique element in $C$ such that $\dst(z,C)=\|z-P_C(z)\|$. 
	
	\subsection{Tools from variational analysis and generalized differentiation}

	For a given closed subset $C \subseteq \H$, the regular (Fréchet) and the basic (limiting or Mordukhovich) normal cone to $C$ at $x$ are denoted and defined respectively by
	\begin{equation*}
	    \nfrech_C(x):=\left\{
	    x^\ast\in\H^\ast\mid\limsup_{x'\xrightarrow{C} x}\frac{\langle x^\ast,x'-x\rangle}{\|x'-x\|}\leq 0
	    \right\}
	\end{equation*}
	and
	\begin{equation*}
	    \nmord_C(x):=\left\{
	    x^\ast\in\H^\ast\mid \exists x_k\xrightarrow{C} x,\,\exists x_k^\ast\rightharpoonup x^\ast:\,x_k^\ast\in \nfrech_C(x_k)
	    \right\}.
	\end{equation*}
	The (effective) domain and epigraph of a function $f:\H\to\Re\cup\{+\infty\}$ are the sets
	\begin{equation*}
	\dom f :=\{x\in\H:f(x)<+\infty\}
	   \text{ and }
	   \,
	   \epi(f):=\left\{(x,\alpha)\in\H\times \Re:f(x)\leq \alpha\right\}
	\end{equation*}
	respectively. For a function $f$ with closed epigraph (i.e., $f$ is lower semicontinuous) its regular  (Fréchet) and basic (limiting or Mordukhovich) subdifferentials at $x\in \H$ may be defined through the corresponding normal cones at its epigraph (see, e.g.,\cite{Mordukhovich_2006,Mordukhovich_2006b,Mordukhovich_2018,Rockafellar_Wets_2009}), or more explicitly, they can be represented as  
	\begin{equation*}
	    \pfrech f (x)=
	    \left\{x^\ast\in \H^\ast\mid \liminf_{x'\to x}\frac{f(x')-f(x)-\langle x^\ast,x'-x\rangle}{\|x'-x\|}\geq 0\right\}
	\end{equation*}
	and 
	\begin{equation*}
	    \pmord f (x):=
	    \left\{x^\ast\in \H^\ast\mid \exists x_k\xrightarrow{f} x,\,\exists x_k^\ast\rightharpoonup x^\ast:\, x_k^\ast \in \pfrech f(x_k) \right\},
	\end{equation*}
	respectively.
	
	Recall that a set-valued mapping $\S:\H\tto\Re^m$ is a mapping that attributes to each $x\in \H$ a subset $\S(x)\subset \Re^m$ and is uniquely defined by its graph
	\begin{equation*}
	    \gph \S:=\{(x,z)\in \H\times\Re^m:z\in \S(x)\}.
	\end{equation*}
	For a set-valued mapping $\S:\H\tto\Re^m$ with closed graph we define its coderivative at $(x,z)\in \gph \S$ as the set-valued mapping $D^\ast \S(x,z):\Re^m\tto \H^\ast$ such that
	\begin{equation*}
	    D^\ast \S(x,z)(z^\ast):=\left\{
	    x^\ast\in\H^\ast\mid (x^\ast,-z^\ast)\in \nmord_{\gph \S}(x,z)
	    \right\}.
	\end{equation*}

	Let $\S :\H \tto \mathbb{R}^{m}$ be a set-valued mapping and $(x,z)\in\gph \S$. We say that $\S$ is locally Lipschitz-like (has the Aubin property) around $(x,z)$ if there exist $\kappa>0$ and $\delta>0$ such that
	    \begin{equation*}
	        \dst(z',\S(x'))\leq \kappa\|x'-x''\|,\,\forall x',x''\in \mathbb{B}_\delta(x)\text{ and } z'\in\S(x'')\cap \mathbb{B}_{\delta}(z).
	    \end{equation*}


\subsection{Assumptions on the underlying set-valued mappings and consequences thereof}

Throughout the manuscript we will assume that given a point of interest $\bar{x}$ the   finite family of set-valued mappings $(\mathcal{S}_i)_{i=1}^s$, which define the probability function \eqref{P00}, satisfies following basic assumptions:
\begin{assumptions}
Let $\bar x \in \H$ be given. There exists a neighborhood $U$ of $\bar x$  such that  for all $i=1,\ldots, p$
	\begin{equation}\label{Assump_Setvalued}\tag{H}
		\begin{array}{cl}
			& a)\ 0 \in \S_i(x) \text{ for all } x\in U\\
			& b)\ \S_i \text{ is locally Lipschitz-like around all } (x, z)\in \gph \S_i  \text{ and } x \in U \\
			& c)\ \S_i \text{ has closed graph and convex values}\\
		\end{array}
	\end{equation}
\end{assumptions}

We can readily observe that when $\bar{z}\in \S_i(x)$ for all $x\in U$ and all $i=1,\ldots,s$ and $\S_i$ satisfies only conditions $b)$ and $c)$ of \eqref{Assump_Setvalued}, that we may consider the set-valued mapping $\tilde{\S}_i(x)=\S_i(x)-\bar{z}$ which will thus satisfy \eqref{Assump_Setvalued}. Condition a) can thus be seen to be a relatively mild condition. 

The next lemma is a technical result which establishes an upper-estimation of the subdifferential of the function $(x,z) \mapsto \frac{1}{2}\dst^{2}(z,\S_i(x)) $ in terms of the coderivative of the mapping $\S_i$.

\begin{lemma}\label{lemma:derivative}
Let $\S_i: \H \tto \Re^m$ be a set-valued mapping satisfying \eqref{Assump_Setvalued} at $\bar x$. Then, the function $u \fc{\H \times \Re^m}{\Re_+}$ given by $u(x,z)=\frac{1}{2}\dst^{2}(z,\S_i(x)) $ is locally Lipschitz around any $(x,z)\in U \times \mathbb{R}^m$. Moreover, for all $(x,z) \in U \times \mathbb{R}^m$, the following subdifferential estimate is valid:
\begin{equation}\label{eq:derivative}
		\pmord u(x,z)\subseteq
		    D^\ast \S(x,P_{\S_i(x)}(z) )( P_{\S_i(x)}(z)-z)\times \{z-P_{\S_i(x)}(z)\}.
\end{equation}
\end{lemma}
	
\begin{proof}
First, let us check the local Lipschitz continuity of $u$ around $(x,z) \in U \times \mathbb{R}^m$. Indeed, on the one hand if $(x,z) \in \gph \S_i$ then by Assumption  \eqref{Assump_Setvalued} we have that $\S_i$ is locally Lipschitz-like around $(x,z)$, hence by \cite[Theorem 1.41]{Mordukhovich_2006}, the distance function $(x,z) \mapsto d(z, \S_i(x))$ is locally Lipschitz. The mapping $u$ is thus also locally Lipschitz around $(x,z)$ as the composition with a continuously differentiable map. 

On the other hand if  $(x,z) \notin \gph \S_i$, we can apply  \cite[Corollary 5.5]{MR2179254} to conclude that $u$ is locally Lipschitz around $(x,z)$.
  	 
Now, let us verify \eqref{eq:derivative}. The function $u$ can be rewritten as a marginal function in the following way:
		\begin{equation*}
		    u(x,z)=\inf\{\psi(x,z,y):y\in \G(x,z)\},
		\end{equation*}
		where $\psi(x,z,y):=\tfrac{1}{2}\Vert z-y\Vert^2$, $\G(x,z):=\S_i(x)$ and where the argminimum for $u$ is the single-valued mapping $M(x,z):=P_{\S_i(x)}(z)$. Since $\S_i$ satisfies \eqref{Assump_Setvalued} at $x$ and $\psi$ is continuously differentiable,  we may apply \cite[Theorem 3.38 i)]{Mordukhovich_2006} to obtain
		\begin{equation*}
		    \pmord u(x,z)\subseteq \bigcup_{(x^*,z^*,y^*)\in\pmord \psi(x,z,P_{\S_i(x)}(z))}(x^*,z^*)+D^*\G(x,z,P_{\S_i(x)}(z))(y^*).
		\end{equation*}
		Finally, the definition of $\G$ implies that $D^*\G(x,z,P_{\S_i(x)}(z))(y^*) = D^*\S_i(x, P_{\S_i(x)}(z))(y^*) \times \col{0}$. This combined with the fact that $\pmord \psi(x,z,P_{\S_i(x)}(z))=(0,z-P_{\S_i(x)}(z),P_{\S_i(x)}(z)-z)$ allow us to rewrite the above inclusion as exhibited in \eqref{eq:derivative}.
	\end{proof}
	The following result can be understood as an interior continuity lemma of the set-valued mappings $(\S_i)_{i=1}^s$ under Assumption \eqref{Assump_Setvalued}.
	\begin{lemma}\label{lemma:inte}
Let $\S_i: \H \tto \Re^m$ be a set valued-mapping satisfying \eqref{Assump_Setvalued} at $\bar x$. Then for every $x\in U$  and $z\in \inte( \S_i(x)) $ there exists $\gamma>0$ such that
		\begin{align*}
			\mathbb{B}_\gamma(z) \subseteq \S_i(x'), \text{ for all  } x'\in \mathbb{B}_\gamma(x).
		\end{align*}
	\end{lemma}
	
	\begin{proof}
Fix $i \in \{1, \ldots,  s\}$, and let $x\in U$  and $z\in \inte( \S_i(x)) $. By the local Lipschitz-like property of $\S_i$ around $(x,z)$ there exist $\epsilon>0$ and $L>0$ such that
    \begin{equation*}
        \S_i(x)\cap \mathbb{B}_\epsilon(z)\subseteq \S_i(x')+L\|x'-x\|\mathbb{B},\,\text{ for all }x'\in \mathbb{B}_{\epsilon}(x).
    \end{equation*}
    Now, by considering $\epsilon>0$ small enough such that $\mathbb{B}_{\epsilon}(z)\subset \S_i(x)$ we may rewrite the above inclusion as
    \begin{equation*}
     \mathbb{B}_\epsilon(z)\subseteq \S_i(x')+L\|x'-x\|\mathbb{B},\,\text{ for all }x'\in \mathbb{B}_{\epsilon}(x).
    \end{equation*}
    Hence, taking $\eta\in (0,\epsilon)$ such that $L\eta<\frac{\epsilon}{2}$ we have that
    \begin{equation}\label{eq:lemma:inte}
     \mathbb{B}_\epsilon(z)\subseteq \S_i(x')+\mathbb{B}_{\frac{\epsilon}{2}}(0),\,\text{ for all }x'\in \mathbb{B}_{\eta}(x).
    \end{equation}
    Let us fix $x'\in \mathbb{B}_{\eta}(x)$. Applying the support function to both sides of \eqref{eq:lemma:inte} and since $\mathbb{B}_\epsilon(z)=\mathbb{B}_{\frac{\epsilon}{2}}(z)+\mathbb{B}_{\frac{\epsilon}{2}}(0)$ we have that for all $h \in \mathbb{R}^m$
    \begin{align*}
     \sigma_{\mathbb{B}_{\frac{\epsilon}{2}}(z)}(h)+\sigma_{\mathbb{B}_{\frac{\epsilon}{2}}(0)}(h) =    \sigma_{\mathbb{B}_\epsilon(z)}(h)&\leq \sigma_{\S_i(x')+\mathbb{B}_{\frac{\epsilon}{2}}(0)}(h),
     \end{align*}
     which therefore implies that 
       \begin{align*}
        \sigma_{\mathbb{B}_{\frac{\epsilon}{2}}(z)}(h)&\leq\sigma_{\S_i(x')}(h)\,\text{ for all }h\in \Re^m.
    \end{align*}
    Due to \cite[Theorem 2.4.14 (vi)]{Zalinescu_2002} together with the convexity of $\S_i(x')$ we obtain that $\mathbb{B}_{\frac{\epsilon}{2}}(z)\subseteq\S_i(x')$. Finally by letting $\gamma=\min\{\eta,\frac{\epsilon}{2}\}$ we conclude the proof.
	\end{proof}
	The next lemma shows that Assumption \eqref{Assump_Setvalued} is sufficient to ensure the continuity of the mapping $(x,z)\mapsto  P_{\S_i(x)}(z)$.
\begin{lemma}
Under Assumption \eqref{Assump_Setvalued} the mappings $(x,z)\mapsto  P_{\S_i(x)}(z)$ are continuous on $U \times \mathbb{R}^m$ for all $i=1,\ldots,s$.
	\end{lemma}

\begin{proof}
Fix $i \in \{1, \ldots, s\}$. Consider a sequence $(x_k,z_k) \to  (x,z )\in  U\times \mathbb{R}^m$. We have that 
	\begin{align*}
		\|  P_{\S_i(x_k)}(z_k) -  P_{\S_i(x)}(z)\| & \leq \|  P_{\S_i(x_k)}(z_k) -  P_{\S_i(x_k)}(z)\| + \|  P_{\S_i(x_k)}(z) -  P_{\S_i(x)}(z)\|\\
	&	\leq  \| z_k -z\| +  \|  P_{\S_i(x_k)}(z) -  P_{\S_i(x)}(z)\|,
	\end{align*}
where in the second inequality, we used the nonexpansiveness of the projection mapping (see, e.g., \cite[p 118]{Urruty_Lemarechal_1996a}). Now, define $y_k:= P_{\S_i(x_k)}(z)$. The sequence $y_k$ can be assumed bounded since $\norm{y_k - z} = \dst(z, \S_i(x_k))$ and the distance function is continuous as a result of Lemma \ref{lemma:derivative}. Hence, it is enough to show that each cluster point of $(y_k)$ is equal to $P_{\S_i(x)}(z)$. Indeed, let $y_{k_l} \to y$.  By closedness of the graph of $\S_i$ we have that $y \in \S_i(x)$. Furthermore, by definition of projection we have that $\|y_{k_l} - z\| = \dst( z, \S_i(x_{k_l}))$ which by continuity of the distance function (Lemma \ref{lemma:derivative}) yields $ \| y - z\| = d(z,\S_i(x))$. Finally, from the uniqueness of the projection onto convex sets we conclude that $y = P_{\S_i(x)}(z)$, and that ends the proof.
\end{proof}

Next, we show that under the Assumption \eqref{Assump_Setvalued} it is possible to find a uniform  lower bound  (with respect to z) for the inner product between $z- P_{\S_i(x)}(z) $ and $z$ around any $x \in U$ in terms of the distance between $z$ and $\S_i(x)$.

	\begin{lemma}\label{lemma:cont:project}
	Consider $\S_i$ satisfying Assumption \eqref{Assump_Setvalued} and suppose that $0 \in \inte(\mathcal{S}_i(\bar{x}))$. Then there exists a neighbourhood of $U'$ of $\bar x$ and $r>0$ such that for all $x\in U'$ we have that
	\begin{align}\label{lemma:cont:project_eq01}
		\langle z- P_{\S_i(x)}(z) 	,z \rangle \geq r \dst (z, \S_i(x)), \text{ for all }(x,z)\in U'\times\mathbb{R}^m.
	\end{align}
\end{lemma}

\begin{proof}
By Lemma \ref{lemma:inte} we can consider a neighbourhood $U'$ of $\bar{x}$ and $r>0$ be such that $r\mathbb{B} \subseteq \S_i(x)$  for all  $x \in U'$. On the one hand if $z\in \S_i(x)$ the inequality holds trivially. Let us now define $y:=r\frac{z -  P_{\S_i(x)}(z)}{ \| z -  P_{\S_i(x)}(z)\| }$ and observe $y \in r\mathbb{B} \subseteq \S_i(x)$. Then, on the other hand, for that $y$:
	\begin{align*}
		\langle 	z- P_{\S_i(x)}(z) ,z \rangle &= 	\langle 	z- P_{\S_i(x)}(z) ,z -P_{\S_i(x)}(z)  \rangle  +\langle 	z- P_{\S_i(x)}(z) , P_{\S_i(x)}(z)  \rangle \\
		&= 	\| 	z- P_{\S_i(x)}(z) \|^2  +\langle 	z- P_{\S_i(x)}(z) , P_{\S_i(x)}(z) -y \rangle + \langle 	z- P_{\S_i(x)}(z) ,  y  \rangle
		\\
		&= 	\| 	z- P_{\S_i(x)}(z) \|^2  +\langle 	z- P_{\S_i(x)}(z) , P_{\S_i(x)}(z) -y\rangle + r	\| 	z- P_{\S_i(x)}(z) \| .
		\\
	\end{align*}
	We notice that  $\langle 	z- P_{\S_i(x)}(z) , P_{\S_i(x)}(z) -w \rangle \geq 0$ for any $w \in \S_i(x)$ by definition of the projection onto convex sets. As a result, it follows that
	\begin{align*}
		\langle z- P_{\S_i(x)}(z) ,z\rangle &\geq  r	\| 	z- P_{\S_i(x)}(z)\|=  r\dst(z,\S_i(x)),
	\end{align*}
	thus concluding the proof.
\end{proof}
To end this subsection, let us show that the non-emptiness of the interior of the constraint set $\mathcal{S}_i(x)$ implies the nontriviality of the normal cone and a``transversality" condition, the importance of which will become clear later.

\begin{lemma}\label{Transversal_Directions}
Under the assumptions of Lemma \ref{lemma:cont:project}, there exists a neighbourhood of $U'$ of $\bar x$ and $r>0$ such that for all $x\in U'$ and all $z \in \bd (\S_i(x))$,  $\nmord_{\S_i(x)}(z) \cap \sph \neq \emptyset$ and $\langle v^\ast, z\rangle >r $ for all $v^\ast \in \nmord_{\S_i(x)}(z) \cap \sph$.
\end{lemma}
\begin{proof}
Using \cite[Corollary 2.24]{Mordukhovich_2006} we have that $\nmord_{\S_i(x)}(z) \cap \sph\neq \emptyset$  for every $z \in \bd (\S_i(x))$. Now, using Lemma \ref{lemma:cont:project} let us consider $U'$ of $\bar x$ and $r>0$ such that \eqref{lemma:cont:project_eq01} holds.   Now, let us consider $ v^\ast \in \nmord_{\S_i(x)}(z) \cap \sph$, so using the finite-dimensional characterization of normal vectors given in \cite[Theorem 1.6]{Mordukhovich_2006}, we can find a sequence $z_k \to z$  such that $(z_k-P_{\S_i(x)}(z_k))/\dst (z_k, \S_i(x)) \to v^\ast $. Hence,    \eqref{lemma:cont:project_eq01} implies that $ \langle v^\ast, z\rangle >r $ , which ends the proof. 
\end{proof}

\subsection{Spherical radial decomposition for set-valued inclusion}
	In this section, we present a spherical radial decomposition for the probability function $\varphi$ given in \eqref{P00}. First, let us suppose that $\xi$  has a density called $f_{\xi}$ with respect to the $m$-dimensional Lebesgue measure $\lambda_m$. Using  this fact  it is clear that $\phi$ can also be
	represented as follows:
	\begin{equation}\label{P01}
		\varphi(x)= \int\limits_{\{ z  \in \Re^m\; :\; z\in \S_i(x)\; \forall i=1,\ldots,s \}   }
		f_\xi (z)  d \lambda_m(z).
	\end{equation}
	Moreover, throughout this
	work we will make the assumption that $f_{\xi}$ is bounded on
	compact sets, i.e.,
	\begin{align}\label{eq:boundedassump:set}
		f_\xi \in  L^\infty(K), \text{ for  every compact set  } K
		\subseteq \Re^{m}. 
	\end{align}
	It has been shown in \cite{vanAckooij_Perez-Aros_2022} that  \eqref{P01} can be represented as 
	\begin{equation*}
		\varphi(x)= \int\limits_{\mathbb{S}^{m-1}} e(x,v)  d
		\mu_\zeta(v),
	\end{equation*}
	where $e: \H \times \mathbb{S}^{m-1} \to \mathbb{R}\cup
	\{ \infty\}$ is the radial ``probability-like" function given
	by
	\begin{align}\label{definition_probability-like_e}
		e(x,v) =  \frac{2 \pi^{\frac{m}{2} } |\det(L)|}{ \Gamma(
			\frac{m}{2} )     } \displaystyle  \int\limits_{  \{ r\geq 0\; :\;
			 rLv  \in \S_i(x)\; \forall i=1,\ldots,s      \}      }   r^{m-1}f_\xi (rLv)
		dr.
	\end{align}
	
	In order to simplify the notation let us define the
	``density-like" function $\theta$:
	\begin{align}\label{transformationftorho}
		\theta(r,v):=   \frac{2\pi^{\frac{m}{2}} |\textnormal{det}(L)|
		}{\Gamma(\frac{m}{2})}  r^{m-1}f_\xi (rLv).
	\end{align}
	
Following \cite{vanAckooij_Perez-Aros_2022}, we define the set-valued mapping $I_\theta : \Re_+ \times
	\sph \tto \Re_+$ by
	\begin{equation}\label{defin_Itheta}
		I_{\theta}(r,v) := [\underline{\theta}(r,v),\overline{\theta}_+(r,v)]
		\cup[\underline{\theta}^-(r,v),\overline{\theta}(r,v) ].
	\end{equation}
	where
	\begin{equation}\label{defin_thetafunctions}
		\begin{aligned}
			\overline{\theta}(r,v)&:=\inf\left\{ k  > 0 : \exists \epsilon >0 \text{ such  that } \theta(u,v) \leq k \text{ a.e. for }  u \in  [r-\epsilon,r+\epsilon]           \right\},\\
			{\bar{\theta}}_{+}(r,v)&:=\inf\{ k  > 0 : \exists \epsilon >0 \text{ such  that } \theta(u,v) \leq k \text{ a.e. for }  u \in  [r,r+\epsilon]           \},\\
			\underline{\theta}(r,v)&:=\sup\{ k  > 0 : \exists \epsilon >0 \text{ such  that } \theta(u,v) \geq k \text{ a.e. for }  u \in  [r-\epsilon,r+\epsilon]          \},\\
			\underline{\theta}^{-}(r,v)&:=\sup\{ k  > 0 : \exists \epsilon >0 \text{ such  that } \theta(u,v) \geq k \text{ a.e. for }  u \in  [r-\epsilon,r]           \}.\\
		\end{aligned}
	\end{equation}

\section{Inner enlargement of the probability  function}
\label{section:inner_enlargement}
	In order to understand the first order variations of the probability function \eqref{P00} we will consider an inner enlargement of the constraint set where the probability is taken. Formally,  given $\epsilon\geq 0$ we define the probability function 
	\begin{align}\label{P01eps} 
		\varphi_\epsilon(x)&:=\mathbb{P} (\omega \in \Omega:  \xi(\omega) \in \S_i(x)+  \epsilon \mathbb{B} \text{ for all }i=1,\ldots,s).
	\end{align}
	Here, it is important to notice that $z \in \S_i(x)+  \epsilon \mathbb{B}$ if and only if $\dst(z, \S_i(x)) \leq \epsilon$. Therefore,  the probability function \eqref{P01eps} can be rewritten using the distance function as
	\begin{align*}
		\varphi_\epsilon(x)=\mathbb{P} (\omega \in \Omega: \dst(\xi(\omega), \S_i(x)) \leq \epsilon \text{ for all }i=1,\ldots,s).
	\end{align*}

	
	Now, consider a point $\bar{x}$ where Assumption \eqref{Assump_Setvalued} holds, and let $U$ be the neighbourhood, where this assumption is satisfied.  We define on $U$  the sets of 
	finite and infinite directions with respect to $\S_i$ as
	the sets defined by
	\begin{align}
		\mathcal{F}_i(x) &:=\{v\in\mathbb{S}^{m-1}\ |\ \exists r> 0: \dst(rLv,\S_i(x))>0\}, 
		\label{finite_directions_Si}\\
		\mathcal{I}_i(x) &:=\{v\in \mathbb{S}^{m-1}\ |\ \forall r\geq 0: \dst(rLv,\S_i(x))=0\},
		\label{infinite_directions_Si}
	\end{align}%
	respectively. We also consider 
	\begin{align}
		\mathcal{F}(x) &:=\{v\in \mathbb{S}^{m-1}\ |\ \exists r> 0: \max_{1\leq i\leq s}\dst(rLv,\S_i(x))>0\}, 
		\label{finite_directions_S}\\
		\mathcal{I}(x) &:=\{v\in \mathbb{S}^{m-1}\ |\ \forall r\geq 0: \max_{1\leq i\leq s}\dst(rLv,\S_i(x))=0\}.
		\label{infinite_directions_S}
	\end{align}%

Let us introduce the radial functions associated with the spherical radial decomposition of our enlargement \eqref{P01eps}. 	Given $\epsilon\geq 0$ we define $\rho^\epsilon_i  : U \times \mathbb{S}^{m-1} \to \Re\cup\{ + \infty \} $   by
\begin{align}
	\label{radial_function_enlargement_i}
		\rho^\epsilon_i \left( x,v\right) &:=\sup\left\{  r>0 : \dst(rLv,\S_i(x))\leq \epsilon \right\}
	\end{align}
and set $\rho_\epsilon(x,v)=\min_{1\leq i\leq s}\rho^\epsilon_i(x,v)$. Particularly, we simply denote $\rho_i \left( x,v\right) :=\rho^0_i \left( x,v\right)$ and $\rho \left( x,v\right) :=\rho_0 \left( x,v\right)$. Now, let us establish some basic properties and relations of the radial functions, which allow us a better understanding of the behaviour of these mappings.
	
\begin{lemma}\label{domain:rho}
Let each $\S_i$ of the of the family of set-valued mappings satisfy \eqref{Assump_Setvalued} at $\bar x$ and let $U$ be the common neighbourhood. Then, we have that:
		
\begin{enumerate}[label=\alph*)]
\item  For all $\epsilon \geq 0$ and all $x\in U$, $\mathcal{F}_i(x)= \dom{\rho^\epsilon_i(x,\cdot)} $ and 
			$\mathcal{F}(x)=\dom{\rho_\epsilon(x,\cdot)}$.
			\item For all $v\in  \mathcal{F}_i(x)$, 
			\begin{align}\label{increasing}
				\dst(r_1Lv, \S_i(x)) <	\dst(r_2Lv,\S_i(x)), \text{ for all } r_2>r_1 >\rho_i(x,v).
			\end{align}
			\item  For all $v\in  \mathcal{F}_i(x)$, $\lim\limits_{r\to \infty}\dst(rLv,\S_i(x)) =+\infty$.
			\item For all $v\in  \mathcal{F}_i(x)$ and all   $r>\rho^\epsilon_i(x,v)$ we have  $\dst(r Lv,\S_i(x))>\epsilon$. 
			\item For all $\epsilon\geq 0$, we have $\rho^\epsilon_i(x,v) =\inf\{ r>0 : \dst(r Lv,\S_i(x))>\epsilon  \}$ with the convention $\inf \emptyset =+\infty$.
			\item  For all $v\in  \mathcal{F}_i(x)$, and all $\epsilon>0$,  $\rho^\epsilon_i(x,v)$ is the unique $r>0$ such that $\dst(r Lv,  \S_i(x))=\epsilon$.
		\end{enumerate}	  
\end{lemma}
	
\begin{proof}
Let us first prove $b)$. Fix $i$, let $r,\beta \in \mathbb{R}$  be given and consider the function \begin{align*}
\gamma_{ r,\beta}(t):= \dst( (t+r)Lv,\S_i(x)) - \beta \text{ for } t\geq 0.
\end{align*}
It is easy to see that $\gamma_{ r,\beta}$ is a  convex function, and therefore, whenever $\gamma_{ r,\beta}(0)<0$ and $\gamma_{ r,\beta}(t_2)\geq 0$, we have that  
\begin{align}\label{inequa:01}
\gamma_{ r,\beta}(t_1) <\frac{ t_1}{t_2} \gamma_{ r,\beta}(t_2), \text{ for all } 0< t_1 < t_2.  
\end{align}
We will make use of this inequality by picking appropriate $r,\beta, t_1$ and $t_2$ values. Indeed,  if $v\in \mathcal{F}_i(x)$, then, for some $r'>0$, we have that $\dst(r'Lv,\S_i(x))>0$. Hence, by convexity of the distance function we have that $r_0 :=\rho_i(x,v)<r'<+\infty$. Now, consider $r_2 >r_1 >r_0$ and fix $\beta\in (0,\dst(r_1Lv, \S_i(x)))$. Then, using inequality \eqref{inequa:01} with $t_1 =r_1 -r_0$, $t_2=r_2 -r_0$ and $r=r_0$ we see that for all $r_2>r_1 >\rho_i(x,v)$
\begin{align}\label{increasing_general}
	0 < \dst(r_1Lv, \S_i(x)) -\beta <\frac{r_1-\rho_i(x,v)}{r_2 - \rho_i(x,v)  }\left( 	\dst(r_2Lv,\S_i(x)) -\beta \right) < 	\dst(r_2Lv,\S_i(x)) -\beta,
\end{align}
which in fact  implies that \eqref{increasing} holds, thus proving part $b)$. Now, since 
$\beta\in (0,\dst(r_1Lv, \S_i(x)))$, letting $r_2 \to \infty$ in \eqref{increasing_general} we have that necessarily $\lim_{r_2\to \infty}\dst(r_2Lv,\S_i(x)) =+\infty$,  which shows $c)$. Now, let us prove $a)$. If $v\in\mathcal{F}(x)$ then $v\in \mathcal{F}_i(x)$ for some $i\in\{1,\ldots,s\}$ and by Item $b)$ and $c)$ the set $\{r\geq 0: \dst(rLv,\S_i(x))\leq \epsilon\}$ must be bounded yielding $\rho_i^\epsilon(x,v)<+\infty$ and in consequence $\rho_\epsilon(x,v)<+\infty$. On the other hand if $v\in \mathcal{I}(x)$ we have that $\rho_i^\epsilon(x,v)=+\infty$ for all $i$ and so $\rho_\epsilon(x,v)=+\infty$ concluding the proof of Item $a)$. Item $d)$ follows by using \eqref{increasing} with $r_1=\rho_i^\epsilon(x,v)$ and $r_2=r$. Item $e)$ follows from Item $d)$ and the continuity of the distance function. Finally, Item $f)$ follows from Items $d)$ and $e)$.
\end{proof}

The next definition corresponds to a growth condition over the coderivative   of the set-valued mappings $\S_i$. It corresponds to a natural extension of the growth condition used in the (sub-)gradient formul{\ae} obtained using the spherical radial decomposition (see, e.g., \cite{vanAckooij_Henrion_2016}). Furthermore, this growth condition is a technical assumption, which can be easily verified in many applications, and allows us to bound the subgradients of the radial “probability-like” function, defined in \eqref{definition_probability-like_e}, locally around $\bar{x}$ and uniformly over all directions $v\in \mathbb{S}^{m-1}$. 
	
\begin{definition}[$\eta$-growth condition for a family of set-valued mappings]
\label{def:growth:setvalued} 
Consider $\bar{x} \in U$ fixed but arbitrary.  Let $\eta : \Re \to  [0,+\infty]$  be a non-decreasing mapping such that
\begin{equation*}
		\lim\limits_{\|z\|\to+\infty}  \|z\|^{m} \bar{f}_\xi (z) \eta(\|z\|) =0.
\end{equation*}

We say that the family of set-valued mappings $\S_i$ satisfies the $\eta$-growth condition at $\bar{x}$ if for some $l>0$
\begin{equation}\label{thetagrowth2:setvalued}
			\| D^\ast \S_i(x,z)\| \leq l \eta(\|z\|),  \;
			\forall x \in \mathbb{B}_{1/l} (\bar{x}), \;\; \forall z\in\Re^m,
\end{equation}
	
where \begin{equation*}
\|D^\ast \S_i(x,z) \|:=\sup\left\{ \| x^\ast\| : x^\ast \in D^\ast \S_i(x,z)(z^\ast) \text{ and } \| z^\ast\| = 1 \right\}. 
\end{equation*}
\end{definition}

Here, it is important to mention that as a result of the coderivative    being a positively homogeneous set-valued mapping, we can deduce the following inequality:
\begin{equation*}
\|D^\ast \S_i(x,z)(z^*)\| \leq  \|D^\ast \S_i(x,z) \| \|z^*\|, \text{ for any } z^*\in \mathbb{R}^m\backslash\{0\}. 
\end{equation*}
 Let us define, for $x\in U$, $\epsilon >0$ and $v\in \mathcal{F}(x)$ 
	\begin{align}\label{def:Meps}
		\mathcal{M}_\epsilon(x,v):=\left\{ \alpha   \cdot \frac{  x^\ast }{%
			\left\langle   \zep{\epsilon}{x}{v}- P_{\S_i(x)}(\zep{\epsilon}{x}{v})
			,Lv\right\rangle }      : \begin{array}{c}
			i\in T^\epsilon_x(v),\alpha \in I_{\theta}(\rho_\epsilon(x,v),v)\\
			x^\ast \in  D^\ast \S_i(x,P_{\S_i(x)}(\zep{\epsilon}{x}{v}) )(  P_{\S_i(x)}(\zep{\epsilon}{x}{v}) -\zep{\epsilon}{x}{v} )
		\end{array} \right\}   
	\end{align}
	where  
	\begin{equation}\label{definitionzep}
	 \zep{\epsilon}{x}{v}:=\rho_\epsilon(x,v)Lv \text{ and }    T_{x}^\epsilon(v)=\left\{i\in \{1,\ldots,s\}:\rho_i^\epsilon(x,v)=\rho_\epsilon(x,v)\right\}.
	\end{equation}
	For convenience, we define $  \mathcal{M}_\epsilon(x,v)= \{ 0\}$ for all  $v\in \mathcal{I}(x)$.

 The following result corresponds to the main theorem of this section, which establishes the Lipschitz continuity and a subdifferential formula for the enlarged probability function \eqref{P01eps}. The result essentially follows from an application of Theorem 3.1 in \cite{vanAckooij_Perez-Aros_2022}, except that the latter result is given in finite dimension and with smooth constraint system only. The appendix therefore provides a careful extension to the infinite dimensional setting and proof of the Theorem.

\begin{theorem}\label{Theorem:epsilon}
Let $\bar{x} \in U$ be given and assume that \eqref{eq:boundedassump:set} holds true. Moreover, assume that the family of set-valued mappings $\S_i$ satisfies the $\eta$-growth condition at $\bar{x}$ and that each $\S_i$ satisfies Assumption  \eqref{Assump_Setvalued} at $\bar x$ with common neighbourhood $U$. Then the probability function \eqref{P01eps} is locally Lipschitz around $\bar x$ and on an appropriate neighbourhood $U'$ of $\bar x$ it holds:
		\begin{align}
		\label{eq:pmord_inclusion}
			\pmord \varphi_\epsilon (x)\subseteq -\cl{{w}^\ast }\left(  \int_{\mathbb{S}^{m-1}}  \mathcal{M}_\epsilon(x,v)  d\mu_\zeta(v)\right), \text{ for all } x \in U'
		\end{align}
 where $\mathcal{M}_\epsilon(x,v)$ is defined in \eqref{def:Meps}. In addition, if $\H$ is finite-dimensional the closure can be omitted in the right-hand side of \eqref{eq:pmord_inclusion}.
\end{theorem}

\section{Lipschitz continuity and subdifferential of probability function $\varphi$}
\label{section:main}
Now, we will focus on establishing the Lipschitz continuity and subgradient formula for the probability function $\varphi$ given in \eqref{P00}. Our technique of proof relies on the use of the subgradient estimation of the enlargement $\varphi_\epsilon$, established in Theorem \ref{Theorem:epsilon}, to approximate the subdifferential of the probability function $\varphi$. The derivation of the main result will be the topic of the first paragraph in this section. Subsequently, we will consider two derived settings and present how the main result can be declined in such cases.

\subsection{Subgradient estimate for the probability function}

The next lemma provides a first estimation of the subdifferential of $\phi$, leveraging on our earlier obtained results for the mapping $\phi_{\eps}$.
	\begin{lemma}[Approximation of subgradients]
	\label{aprox:subgrad}
		Consider the probability function $\varphi$ defined in \eqref{P00}, and the family of probability functions $\varphi_\epsilon$ given by \eqref{P01eps}. 
		Then, for all $x\in U$ 
		\begin{align}\label{eqinf}
			\varphi(x)&=\inf\limits_{\epsilon>0} \varphi_\epsilon(x).\\
		\pmord\varphi(x)& \subseteq \left\{  x^\ast \in \H :  \begin{array}{c}
		\text{There exist } x_k \to x, \epsilon_k \to 0^+ \\
		\text{and  } x_k^\ast \in \pmord\varphi_{\epsilon_k}(x_k)  \text{ s.t. }  x_k^\ast \rightharpoonup x^\ast 
		\end{array}   \right\}.	\nonumber	
		\end{align}
	%
	\end{lemma}
	
\begin{proof}
	A direct application of the continuity of the probability measure shows  \eqref{eqinf}. Now, consider a point  $x^\ast \in \pmord \varphi(x)$. It follows from definition that there are sequences $x_k \to x$ with $\varphi(x_k)\to\varphi(x)$ and $x_k^\ast  \rightharpoonup  x^\ast$ such that $x^\ast_{k} \in \pfrech \varphi(x_k)$. Hence, using \cite[Lemma 2.1]{vanAckooij_Perez-Aros_2019} for each point $x^\ast_{k}$ we can get  sequences $x_{k,j} \to x_k$ with $\varphi(x_{k,j})\to\varphi (x_k)$ and $x_{k,j}^\ast \to x_{k}^\ast$ such that $x_{k,j}^\ast\in  \pfrech \varphi_{\epsilon_{k,j}}  (x_{k,j})$. Now, since $\H$ is reflexive and separable we have that the weak$^\ast$-topology is metrizable on bounded sets (see, e.g., \cite{MR2766381}), and it allows us to use a diagonal argument to conclude the result.
	\end{proof}
	

The following lemma shows that the radial functions generated by the set-valued mappings $\S_i$, defined in \eqref{radial_function_enlargement_i}, are continuous with respect to all the parameters $(\epsilon, x, v)$.
	\begin{lemma}\label{continuity:rhoeps}
	Let each $\S_i$ of the family of set-valued mappings satisfy Assumption \eqref{Assump_Setvalued} at $\bar{x}\in U$, the latter being the common neighbourhood.
	Then, there exists an open neighborhood $U' \subseteq U$, such that for every sequence $[0,+\infty)\times U'\times \mathbb{S}^{m-1}\ni (\epsilon_k,x_k,v_k) \to (\epsilon,  x,v)  \in [0,+\infty)\times U'\times \mathbb{S}^{m-1}$ we have that
		\begin{align}
		\label{eq:continuity:rhoeps_i}
			\rho_i^\epsilon(  x,v)=\lim\limits_{ k \to +\infty} \rho_i^{\epsilon_k}(x_k,v_k) \text{ for each }i=1,\ldots,s.
		\end{align}
	Furthermore,
	\begin{align}
	\label{eq:continuity:rhoeps}
			\rho_\epsilon(  x,v)=\lim\limits_{ k \to +\infty} \rho_{\epsilon_k}(x_k,v_k).
		\end{align}
	\end{lemma}
	
	\begin{proof}
	Fix $i$ and consider $(\epsilon_k,x_k,v_k)\to (\epsilon,x,v)$. Let us first prove equality \eqref{eq:continuity:rhoeps_i} by assuming that the sequence $\rho_i^{\epsilon_k}(x_k,v_k)$ diverges. Suppose by contradiction that $\rho_i^\epsilon(x,v)<+\infty$ and consider $r>\rho_i^\epsilon(x,v)$. By Lemma \ref{domain:rho} we have that  $\dst(rLv,\S_i(x)) >\epsilon$, so by continuity of the distance function (recall Lemma \ref{lemma:derivative}) we have that $\dst(rLv_k,\S_i(x_k))>\epsilon_k$ for $k$ large enough, which, again by Lemma \ref{domain:rho}  means that $r\geq \rho_i^{\epsilon_k}(x_k,v_k)$ for $k$ large enough. In other words $\rho_i^{\epsilon_k}(x_k,v_k)$ is bounded from above and does not diverge, thus contradicting the earlier assumption.
	
	Thus, we assume that the sequence $\rho_i^{\epsilon_k}(x_k,v_k)$ admits a cluster point $r'$. Then for some subsequence we have that $\rho_i^{\epsilon_{k_l}}(x_{k_l},v_{k_l})\to_l r'$. Let us prove that $r'=\rho_i^\epsilon(x,v)$. By Lemma \ref{domain:rho} we have the equality $\dst(\rho_i^{\epsilon^{k_l}}(x_{k_l},v_{k_l})Lv_{k_l},\S_i(x_{k_l}))=\epsilon_{k_l}$ which by continuity of the distance function yield us the equality $\dst(r'Lv,\S_i(x))=\epsilon$. From the uniqueness shown in Lemma \ref{domain:rho} the result for the case $\epsilon>0$ follows. Thus it remains to prove it for $\epsilon=0$. In this case, by definition of the radial function we have that $r'\leq \rho_i(x,v)$. Suppose by contradiction that $r'<\rho_i(x,v)$. Therefore, $r'Lv \in \inte(\S_i(x))$ and hence $\rho_i^{\epsilon_{k_l}}(x_{k_l},v_{k_l})Lv\in\inte (\S_i(x))$ for $l$ large enough. By Lemma \ref{lemma:inte} we have that there exists $\gamma>0$ such that  $(\rho_i^{\epsilon_{k_l}}(x_{k_l},v_{k_l})+\gamma)Lv  \in \inte(\S_i(x_k))$, which particularly, by definition of the radial function,  implies that  $\rho_i^{\epsilon_{k_l}}(x_{k_l},v_{k_l})+\gamma \leq \rho_i (x_{k_l},v_{k_l})$ for $l$ large enough, which is a contradiction. Therefore, we have that $r'=\rho_i(x,v)$ and since this holds true for all possible cluster points we conclude \eqref{eq:continuity:rhoeps_i}.
	
	Now let us prove \eqref{eq:continuity:rhoeps}. On the one hand, there is some $i\in\{1,\ldots,s\}$ such that $\rho_\epsilon(x,v)=\rho_i^\epsilon(x,v)$ which together with \eqref{eq:continuity:rhoeps_i} lead us to $\rho_\epsilon(x,v)=\lim_k \rho_i^{\epsilon_k}(x_k,v_k)$. Since, by definition, $\rho_i^{\epsilon_k}(x_k,v_k)\geq \rho_{\epsilon_k}(x_k,v_k)$ for all $k$, we conclude that $\rho_\epsilon(x,v)\geq\lim_k \rho_i^{\epsilon_k}(x_k,v_k)$. On the other hand, for each $k$ there exists $i_k\in\{1,\ldots,s\}$ such that $\rho_{\epsilon_k}(x_k,v_k)=\rho_{i_k}^{\epsilon_k}(x_k,v_k)$. Under subsequence, we may assume that $\rho_{\epsilon_k}(x_k,v_k)=\rho_{i}^{\epsilon_k}(x_k,v_k)$ for some fixed $i$. By taking limits on this last equality and thus by \eqref{eq:continuity:rhoeps_i} we obtain that $\lim_k\rho_{\epsilon_k}(x_k,v_k)=\rho_i^\epsilon(x,v)$ which by definition of $\rho_\epsilon(x,v)$ lead us to $\lim_{k}\rho_{\epsilon_k}(x_k,v_k)\geq\rho_{\epsilon}(x,v)$, concluding the proof of \eqref{eq:continuity:rhoeps} and thus of the lemma.
\end{proof}
	
Now, the following lemma provides an upper-estimation of the set-valued mappiong $\mathcal{M}_\epsilon(x,v)$ locally around $x$ and uniformly with respect to $v\in \mathbb{S}^{m-1}$ and $\epsilon >0$.
\begin{lemma}\label{bounded:lemma}
	Let each $\S_i$ of the family of set-valued mappings satisfy Assumption \eqref{Assump_Setvalued} at $\bar{x} \in U$ with $0\in\inte(\S_i(x))$ for all $x\in U$, where $U$ is a common neighbourhood. Moreover, assume that the family of set-valued mappings satisfies the $\eta$-growth condition at $\bar{x}$ and that \eqref{eq:boundedassump:set} holds true.
	
	Then,	there  exist  a neighbourhood $U'$ of $\bar x$ and $\epsilon'>0$ such that 
		\begin{align*}
			\sup\{ \|x^\ast\|: x^\ast \in \mathcal{M}_\epsilon(x,v)  ,\; v\in \mathbb{S}^{m-1},\;x\in U',\; \epsilon \in (0,\epsilon')  \}<\infty.
		\end{align*}
			Moreover, $\limsup\limits_{(\epsilon,x,v)\to (0,\bar{x},\bar{v})} \mathcal{M}_\epsilon(x,v) = \{ 0\}$ for all $\bar v\in \mathcal{I}(\bar x)$.
	\end{lemma}
	
	\begin{proof}
		Due to the compactness of $\mathbb{S}^{m-1}$ it is enough to show that for all $v\in \mathbb{S}^{m-1}$ 	there  are    neighbourhoods $U_v$ of $\bar x$, $V_v$ of $v$, and $\epsilon_v>0$ such that 
		\begin{align}\label{localbound}
			\sup\left\{ \|x^\ast\|: x^\ast \in \mathcal{M}_\epsilon(x,v)  ,\; v\in V_{v},\;x\in U_{v},\; \epsilon \in (0,\epsilon_v')  \right\}<\infty.
		\end{align}
		Fix $\bar v\in \mathbb{S}^{m-1}$ and let us suppose first that $\bar{v}\in \mathcal{F}(\bar{x})$. By Lemma \ref{continuity:rhoeps}, there  are neighbourhoods $U_{\bar v}$ of $\bar x$, $ V_{\bar v}$ of $\bar v$, and $\epsilon_{{\bar v}}>0$ such that
		$$\sup\{ \rho_{\epsilon}(x,v) : (\epsilon,x,v)\in W:= [0,\epsilon_{\bar v}] \times U_{\bar v} \times V_{\bar v}\} <+\infty.$$ 
		Particularly, $v\in \mathcal{F}(x)$ for all $(x,v) \in U_{\bar v}\times V_{\bar v}$. Now, fix $(\epsilon,x,v)\in W $ and consider a point $w^\ast \in \mathcal{M}_\epsilon(x,v)$ of the form:
		\begin{align}\label{vector_w}
			\alpha  \frac{  x^\ast }{%
			\left\langle   \zep{\epsilon}{x}{v}- P_{\S_i(x)}(\zep{\epsilon}{x}{v})
			,Lv\right\rangle }, 
		\end{align}
		for some $i\in T_x^\epsilon(v)$, $\alpha \in  I_{\theta}(\rho_\epsilon(x,v),v)$,  
		$x^\ast \in  D^\ast \S_i(x,P_{\S_i(x)}(\zep{\epsilon}{x}{v}) )(  P_{\S_i(x)}(\zep{\epsilon}{x}{v}) -\zep{\epsilon}{x}{v} )$ 	where $ \zep{\epsilon}{x}{v}:=\rho_\epsilon(x,v)Lv $.
		 Since the set $I_{\theta}(\rho_\epsilon(x,v),v)$ remains bounded on $W$, we have that $\alpha$ is uniformly bounded on $W$, let us say by $\bar{\alpha}$. Now, since $\S_i$ has the local Lipschitz-like property around $(x,P_{\S_i(x)}(\zep{\epsilon}{x}{v}))$ we have that there exists $\kappa_i \geq 0$ such that 
		 \begin{align*}
		 	\| x^\ast \| \leq \kappa_i \|   \zep{\epsilon}{x}{v}- P_{\S_i(x)}(\zep{\epsilon}{x}{v})   \|.
		 \end{align*}
		On the other hand, by Lemma \ref{lemma:cont:project} (shrinking enough the neighbourhood $U_{\bar v}$) there exists some $r_i>0$ such that 
	\begin{equation}\label{existence_r}
	    \left\langle   \zep{\epsilon}{x}{v}- P_{\S_i(x)}(\zep{\epsilon}{x}{v})
		,Lv\right\rangle \geq \frac{r_i}{\rho_\epsilon(x,v) }\|   \zep{\epsilon}{x}{v}- P_{\S_i(x)}(\zep{\epsilon}{x}{v})\|.
	\end{equation}

		Therefore, we obtain that $\| w^\ast\| \leq  \frac{ \bar{\alpha}\kappa  \bar{\rho} }{ r} $, where $\kappa:=\max \kappa_i$, $r:=\min r_i$ and $\bar{\rho} :=\sup\{ \rho_\epsilon(x,v) : (\epsilon,x,v)\in W  \}$ which means that \eqref{localbound} holds for  $\bar{v}\in \mathcal{F}(\bar{x})$.
		
	Now, consider the case when $\bar{v}\in \mathcal{I}(\bar{x})$. Let $\gamma>0$ and let $l>0$ be such that the family of $\S_i$ satisfies the $\eta$-growth condition at $\bar x$  (see Definition \ref{def:growth:setvalued}). By Lemma \ref{continuity:rhoeps} we have that there are neighborhoods $U_{\bar{v}}$ of $\bar{x}$, $V_{\bar{v}}$ of $\bar{v}$ and $\epsilon_{\bar{v}}>0$ such that  $\rho_{\epsilon}(x,v)>l$ for all $(\epsilon, x,v )\in W$. Moreover, by Lemma \ref{lemma:cont:project}, there exists $r>0$ such that \eqref{existence_r} holds. Therefore, $w^\ast \in \mathcal{M}_\epsilon(x,v)$ in the form of \eqref{vector_w} satisfies
		\begin{align*}
			\|w^\ast\| & \leq \bar{\theta} (\rho_\epsilon(x,v),v)  \frac{\rho_\epsilon(x,v) }{r \| \zep{\epsilon}{x}{v}- P_{\S_i(x)}(\zep{\epsilon}{x}{v})\|  } l  \eta(\|\zep{\epsilon}{x}{v}\|)  \| \zep{\epsilon}{x}{v}- P_{\S_i(x)}(\zep{\epsilon}{x}{v})\|.
			\end{align*}
			Furthermore, for some constant $C>0$ we have that $\bar{\theta} (\rho_\epsilon(x,v),v)\leq C(2\rho_\epsilon(x,v))^{m-1}\bar{f}_\xi(\zep{\epsilon}{x}{v})$ and thus
			\begin{equation*}
			    \|{w}^*\|\leq \frac{Cl2^{m-1}}{r\|Lv\|^{m-1}}\|\zep{\epsilon}{x}{v}\|^{m}\bar{f}_\xi(\zep{\epsilon}{x}{v})\eta(\|\zep{\epsilon}{x}{v}\|).
			\end{equation*}
		Since, $\rho_\epsilon(x,v) $ can be chosen arbitrarily large (shrinking $W$ if necessary) we can assume that $ \|{w}^\ast\| \leq \gamma$, and that ends the proof.	
	\end{proof}
	
		
%
Now, we are able to show the main result of this work, which provides the locally Lipschitz continuity of the probability function 
\begin{theorem}\label{theo:Main}
		Consider each $\S_i$ in the family of set-valued mappings satisfying Assumption \eqref{Assump_Setvalued} at $\bar{x} \in U$ with $0\in\inte (\S_i(x))$ for all $x\in U$ where the latter is a common neighbourhood. Moreover, assume that the family of set-valued mappings $\S_i$ 
	satisfies the $\eta$-growth condition at $\bar{x}$ and that \eqref{eq:boundedassump:set}  holds true.
		
	Then the probability function \eqref{P00} is locally
		Lipschitz around $\bar x$ and on an appropriate neighbourhood $U'$ of
		$\bar x$ it holds:
		
		\begin{align}
		\label{eq:main:theorem}
			\pmord \varphi (x)\subseteq -\cl{}\left( \int\limits_{\mathcal{F}(x)}  	\mathcal{M}(x,v)    d\mu_\zeta(v) \right), \text{ for all } x \in U',
		\end{align}
		where $\mathcal{M}(x,v)$ is given  as follows
		\begin{align*}
			\mathcal{M}(x,v) = \left\{ \alpha   \cdot \frac{  x^\ast }{%
				\left\langle    z^\ast
				,Lv\right\rangle }      : \begin{array}{c}
				\alpha \in 	I_{\theta}(\rho(x,v),v),
				z^\ast \in \nge_{\S_i(x)}(\zep{}{x}{v})\cap \mathbb{S}^{m-1}  \\ \\i\in T_x(v), 
				x^\ast\in  D^\ast \S_i(x,\zep{}{x}{v} )( -z^\ast) 
			\end{array} \right\} \text{ for   }v\in \mathcal{F}(x)
		\end{align*}
		with
		\begin{equation*}
		\begin{aligned}
 \zep{}{x}{v}:=\rho(x,v)Lv \text{ and }    	    T_{x}(v)&=\left\{i\in \{1,\ldots,s\}:\rho_i(x,v)=\rho(x,v)\right\}\\
	    \end{aligned}
	\end{equation*}
		and by  $\mathcal{M}(x,v)=\{ 0\}$ for   $v\in \mathcal{I}(x)$.
	
		In addition, if $\H$ is finite-dimensional the closure can be omitted in either representation.
\end{theorem}
	
\begin{proof}
Let $U'$ be the neighbourhood of $\bar x$, further subset of $U$ of which the existence is asserted in Lemma \ref{bounded:lemma}. Next, 
let $x \in U'$ be given arbitrarily but fixed and let $x^\ast \in \pmord \varphi (x)$ likewise be arbitrary. We divide the proof into four claims.

 \textbf{Claim 1:} There exist sequences $x_k \to x$,  $\epsilon_k \to 0^+$ and $x_k^\ast \rightharpoonup x^\ast$ with
	$$x_k^\ast \in  -\int\limits_{\mathbb{S}^{m-1}} \mathcal{M}_{\epsilon_k}(x_k,v)    d\mu_\zeta(v),$$ where $\mathcal{M}_{\epsilon_k}$ is as in \eqref{def:Meps}.
%
Indeed,  by Lemma \ref{aprox:subgrad} there exists $x_k^\ast \in \pmord \varphi_{\epsilon_k } (x_k)$ with $x_k \to x$, $\epsilon_k \to 0^+$ and $x_k^\ast \rightharpoonup x^\ast$. Furthermore, by Theorem \ref{Theorem:epsilon},
	\begin{equation*}
	  x^\ast_k \in  C_k:= -\cl{w^\ast}\left(  \int\limits_{ \mathbb{S}^{m-1}}  	\mathcal{M}_{\epsilon_k}(x_k,v)    d\mu_\zeta(v)\right).  
	\end{equation*}
	Let us notice that by Lemma \ref{bounded:lemma} there is some $k_0 \in \mathbb{N}$ such that the  set $\cup_{k \geq k_0} C_k$ is bounded. Since, $\H$ is separable and reflexive, the weak$^\ast$- topology is metrizable on bounded sets (see, e.g., \cite[Theorem 5.1]{Conway_1985}, we also draw attention to \cite[Aufgabe VIII.8.16]{Werner_2011}, showing the necessity of $\H$ being separable) allowing us to take  sequences $x_{j,k}^\ast \rightharpoonup x_k^\ast $ with $x_{j,k}^\ast \in -\int_{\mathbb{S}^{m-1}}  	\mathcal{M}_{\epsilon_k}(x_k,v) d\mu_\zeta(v)$ as well as to use a diagonal argument so we can construct the desired sequence.\\

\textbf{Claim 2: }  There exists a sequence of (Bochner) integrable functions $y_k : \mathbb{S}^{m-1} \to \H$ such that 
	\begin{align*}
	y_k^\ast(v )\in   \mathcal{M}_{\epsilon_k}(x_k,v) \quad \mu_\zeta \text{-a.e.  	and } x_k^\ast = -\int_{ \mathbb{S}^{m-1}} y^\ast_k(v) d\mu_\zeta(v).
	\end{align*}

Using the definition of the integral of a set-valued mapping we get the existence of such sequence.\\

%

 \textbf{Claim 3:} We have that $x^\ast \in -\cl{}\left( \int_{ \mathbb{S}^{m-1}} F (v) d\mu_\zeta(v)\right)$, where $F(v)$ is the set of sequential weak$^\ast$ limits of sequences of the form $\{ y^\ast_k(v)\}$. Moreover, the closure can be omitted when $\H$ is finite-dimensional.

	Indeed, by Lemma \ref{bounded:lemma}, we have that there exists $M>0$ such that   $\|y_k^\ast(v) \|  \leq M$ for almost all $v\in \mathbb{S}^{m-1}$.  Now, by \cite[Corollary 4.1]{MR2197293}, we have that $x^\ast \in -\textnormal{cl}^{w^\ast} \left( \int_{ \mathbb{S}^{m-1}} F (v) d\mu_\zeta(v)\right)$, where the closure operation can be omitted if the space $\H$ is finite-dimensional. Furthermore, by Lyapunov's convexity theorem, we have that the set $\cl{} \left( \int_{ \mathbb{S}^{m-1}} F (v) \mu_\zeta\right) $ is convex, so 
	\begin{equation*}
	   \textnormal{cl}^{w^\ast} \left( \int_{ \mathbb{S}^{m-1}} F (v) d\mu_\zeta\right) =\cl{} \left( \int_{ \mathbb{S}^{m-1}} F (v) d\mu_\zeta\right), 
	\end{equation*}
	and that ends the proof of our claim.\\

 \textbf{Claim 4:} We have that  $F (v) \subseteq  \mathcal{M}(x,v)$ for almost all $v\in \mathbb{S}^{m-1}$.

Indeed, consider a set of full measure $S \subseteq \mathbb{S}^{m-1}$ such that $	y_k^\ast(v )\in   \mathcal{M}_{\epsilon_k}(x_k,v)$ for all $k \in \mathbb{N}$ and $v\in S$. Then, fix  $v\in S$. First, if $v\in \mathcal{I}(x)$, we have  by Lemma  \ref{bounded:lemma} that $F(v) \subseteq \limsup  \mathcal{M}_{\epsilon_k}(x_k,v) \subseteq   \{0\}$. Now, assume that $v\in \mathcal{F}(x)$  and consider $y_v^\ast\in F(v)$. Then there exists a sequence $y^\ast_{k_j}(v)$ such that $y_{k_j}^\ast(v) \rightharpoonup y^\ast_v$ and
	\begin{align*}
	y_{k_j}^\ast(v)= \alpha_j   \cdot \frac{  x_j^\ast }{%
		\left\langle   z_j- P_{\S_{i_j}(x_j)}(z_j)
		,Lv\right\rangle } ,
	\end{align*}
	for some $i_j\in T_{x_j}^{\epsilon_j}(v)$, $\alpha_j \in I_{\theta}(\rho_{\epsilon_j}(x_j,v),v)$ and $x_j^\ast \in  D^\ast \S_{i_j}(x_j,P_{\S_{i_j}(x_j)}(z_j)  )(  P_{\S_{i_j}(x_j)}(z_j) -z_j )$, where $z_j:=\rho_{\epsilon_j}(x_j,v)Lv $. Now, we may assume (by passing to a subsequence), that for some fixed $i\in T_{x_j}^{\epsilon_j}(v)$,
		\begin{align}\label{repr:y}
	y_{k_j}^\ast(v)= \alpha_j   \cdot \frac{  x_j^\ast }{%
		\left\langle   z_j- P_{\S_{i}(x_j)}(z_j)
		,Lv\right\rangle } ,
	\end{align}
	with $\alpha_j \in I_{\theta}(\rho_{\epsilon_j}(x_j,v),v)$ and $x_j^\ast \in  D^\ast \S_{i}(x_j,P_{\S_{i}(x_j)}(z_j)  )(  P_{\S_{i}(x_j)}(z_j) -z_j )$, where $z_j:=\rho_{\epsilon_j}(x_j,v)Lv $.
	
 First we notice that, by Lemma \ref{continuity:rhoeps}, we have $\rho_i(x,v)=\rho(x,v)$, that is, $i\in T_x(v)$. Moreover, the functions $\bar{\theta}$,	$\underline{\theta}$, defined in \eqref{defin_thetafunctions},  are upper semicontinuous and lower semicontinuous, respectively. Therefore, we can assume (by passing to a subsequence) that $\alpha_j \to \alpha \in I(\rho(x,v),v)$.  Now, define $w_j^\ast:= \frac{  z_j- P_{\S_i(x_j)}(z_j) }{\|  z_j- P_{\S_i(x_j)}(z_j)\|   }$, $v^\ast_j :=  \frac{  x_j^\ast }{\|  z_j- P_{\S_i(x_j)}(z_j)\|   }$. Since, $w_j^\ast \in \mathbb{R}^{m}$ and it has unit norm, we can assume that the sequence converges, i.e., $w_j^\ast \to z^\ast $ for some $z^\ast \in \dl{S}^{m-1}$. Now, using the fact that $\S_i$ is locally Lipschitz-like around $(x_j,P_{\S_i(x_j)}(z_j))$, we have that the sequence $v^\ast_j$ is bounded, and from the fact that $\H$ is reflexive, we can assume that $v_j^\ast \rightharpoonup v^\ast$ for some $v^\ast \in \H^\ast$. Let us now prove that $z^\ast \in \nge_{\S_i(x)}(\rho(x,v)Lv)$ and $x^\ast\in  D^\ast \S_i(x,\rho(x,v)Lv)( -z^\ast) $.  Indeed, 
	\begin{enumerate}[label=\roman*)]
		\item $z^\ast \in \nge_{\S_i(x)}(\rho(x,v)Lv)$: Since $\S_i$ is closed and convex valued, we first observe that as a result of the optimality conditions defining the projection, we have 
		$z - P_{\S_i(x)}(z) \in \nge_{\S_i(x)}( P_{\S_i(x)}(z) )$ and as a result $w_j^* \in \nge_{\S_i(x_j)}( P_{\S_i(x_j)}(z_j) )$. 
		We can apply \cite[Corollary 1.96]{Mordukhovich_2006} to get that    $w^\ast_j\in \partial_z d(z_j, \S_i(x_j) )$  for all $j \in \mathbb{N}$. Now we have that for all $j \in \mathbb{N}$ and all  $w \in \mathbb{R}^m$
		
	\begin{align*}
		 \langle w_j^\ast, w- x_j \rangle  \leq d(w ,\S_i(x_j) )  - d(z_j, \S_i(x_j)).
	\end{align*}
		Hence, taking limits in the above inequality and recalling that the distance function is Lipschitz continuous (see  Lemma \ref{lemma:derivative}) we can conclude that   $  z^\ast \in  \partial_z d(\rho(x,v)Lv, \S_i(x ) )$. Finally, using again \cite[Corollary 1.96]{Mordukhovich_2006}, we get that $z^\ast \in  \nge_{\S_i(x)}(\rho(x,v)Lv)$.
		\item $v^\ast\in  D^\ast \S_i(x,\rho(x,v)Lv)( -z^\ast) $: First, by definition of coderivative,  we have that the sequence $(v^\ast_j, w_j^\ast) \in \nmord_{ \gph \S_i} (x_j , z_j)$. Moreover,  by \cite[Theorem 3.60]{Mordukhovich_2006}  the graph of the mapping $(u,v) \to \nmord_{ \gph \S_i}(u,v)$ is locally closed with respect to the  $\| \cdot\| \times w^\ast$-topology at $(x,\rho(x,v)Lv)$ provided that the SNC property holds at $(x,\rho(x,v)Lv)$. Since, $\S_i$ is Lipschitz-like around $(x,\rho(x,v)Lv)$ we can apply \cite[Proposition 1.68]{Mordukhovich_2006} to get that the mapping $\S_i$ is SNC at $(x,\rho(x,v)Lv)$  (see also \cite[Definition 1.67]{Mordukhovich_2006}). Therefore, we have that $(v^\ast, z^\ast) \in \nmord_{ \gph \S_i} (x, \rho(x,v)Lv)$, which by definition  means that $v^\ast\in  D^\ast \S_i(x,\rho(x,v)Lv )( -z^\ast) $.
	\end{enumerate}
Finally, by taking limits on \eqref{repr:y} we get that $y_v^\ast\in \mathcal{M}(x,v)$.

As a result, it is thus clear that  formula \eqref{eq:main:theorem} holds true. 
\end{proof}

\begin{remark}
It is important to emphasize that Lemma \ref{Transversal_Directions}  ensures that for $x$ close enough to $\bar{x}$ the mapping  $\mathcal{M}(x, v)$ is well-defined. Indeed, it excludes that there are some directions $v\in \mathcal{F}(x)$ such that   $\left\langle    z^\ast
				,Lv\right\rangle =0$    for  $
				z^\ast \in \nge_{\S_i(x)}(\rho(x,v)Lv)$.
Such directions are not transversal and constitute a complication as extensively discussed in \cite{vanAckooij_Perez-Aros_2020}. In the classic situation wherein $\S_i(x)$ is given by a set of convex inequalities, together with the assumption $0 \in \inte \S_i(x)$, the existence of non transversal directions can be altogether avoided - see \cite[Lemma 3]{Hantoute_Henrion_Perez-Aros_2017}. In this work, we have also shown that  no such directions exist using advanced tools of variational analysis.
\end{remark}

\begin{example}\label{Example41}
    Let us consider a 2-dimensional Gaussian random vector $\xi \sim \mathcal{N}(0 , I)$ and the set valued mapping $\mathcal{S}: \Re   \tto \Re^2$ given by  \begin{align}\label{MappingS:Example41}
    S(x) :=\{ (z_1,z_2 ) : (z_1+2)( z_2 +2) \geq x,\; z_1 \geq -2 \text{ and }  z_2 \geq -2 \}.
    \end{align}  Consider the probability function generated by this set-valued mapping in the sense of \ref{P00}, that is, $\varphi(x):= \mathbb{P}(\xi \in \mathcal{S}(x))$, 
 and let us examine   at the point  $\bar{x}=1$. We claim that it is possible to  find an appropriate neighbourhood of $\bar{x}$ such that the assumptions of Theorem  \ref{theo:Main}   hold. Indeed, it is easy to see that $\mathcal{S}$ has convex and closed values and $0\in \inte(\mathcal{S}(x))$ (for all $x$ close enough to $\bar{x}$). Moreover,    thanks to \cite[Corollary 4.35]{Mordukhovich_2006}, we can compute the coderivative of $\mathcal{S}$ for $(z_1 + 2) (z_2 + 2)  =x$ as 
    \begin{align}\label{computation_coderivative_example}
        D^\ast \mathcal{S}(x,z_1,z_2)(z^\ast) = \left\{  \frac{ \|z^*\| }{ \sqrt{  (z_1 + 2)^2 + (z_2 +2)^2         }    }  \right\}
    \end{align}
    Furthermore, for $(z_1 +  2) (z_2 + 2)  >x$, we have that both $D^\ast \mathcal{S}(x,z_1,z_2)(0)=\emptyset$ and $D^\ast \mathcal{S}(x,z_1,z_2)(z^\ast)=\emptyset$. It shows the Lipschitz-like property required in Assumption \eqref{Assump_Setvalued}. Moreover, the above computation allows us to estimate $ \|\S(x)\| $. Indeed, consider arbitrary $z^\ast$ with  $\| z^\ast\|=1$ and $x^\ast \in  D^\ast \mathcal{S}(x,z_1,z_2)(z^\ast) $. Hence, by  \eqref{computation_coderivative_example}, we have that for  $(z_1 +  2) (z_2 + 2)  =x$
    \begin{align*}
        \| x^\ast\| =   \frac{ 1 }{ \sqrt{  (z_1 + 2)^2 + (z_2 +2)^2         }    } =  \frac{ 1 }{ \sqrt{  (z_1-z_2)^2    +2x     }    } \leq   \frac{ 1 }{ \sqrt{       x     }    }
    \end{align*}
    Therefore, the $\eta$-growth condition  holds at $\bar{x}=1$. 
\end{example}

\begin{corollary}\label{Corollary:theo:Main}
In the setting of Theorem \ref{theo:Main}, make the following supplementary assumptions:
\begin{enumerate}[label=\alph*)]
    \item the random vector $\xi$ has a continuous density function
    \item the set $T_{\bar x}(v)$ is a singleton $i(v)$ for $\mu_{\zeta}$ almost all $v \in \sph$
    \item the cone $\nge_{\S_{i(v)}(\bar x)}(\zep{}{\bar x}{v})$ contains a unique element $z_{\bar x}^*(v)$ of unit-norm and $D^\ast \S_{i(v)}(\bar x, \zep{}{\bar x}{v})(z_{\bar x}^*(v))$ is reduced to a single element $x_{\bar x}^\ast(v)$ for $\mu_{\zeta}$ almost all $v \in \sph$
    \item the space $\scr{X}$ is finite dimensional.
\end{enumerate}
Then the probability function is differentiable at $\bar x$ and the following formula holds
 \begin{equation*}
        \nabla \phi(\bar x) = \int_{\scr{F}(\bar x)} \theta(\rho(\bar x,v),v)  \frac{x_{\bar x}^\ast(v)}{\langle z_{\bar x}^\ast( v),   L v \rangle   } d\mu_{\zeta}(v). \label{eq:nablaphi}
    \end{equation*}
In addition,     if   the above conditions hold on a neighbourhood, $\phi$ is continuously differentiable on this neighbourhood.
   
\end{corollary}

\begin{proof}
We may employ Theorem \ref{theo:Main} to establish formula \eqref{eq:main:theorem}, without the closure operation since $X$ is finite dimensional. Moreover, the first supplementary assumption entails that $\alpha \in I_{\theta}(\rho(x,v),v)$ is unique and of the form $\alpha = \theta(\rho(x,v),v)$ (see e.g., the text slightly before Proposition 3.2 in \cite{vanAckooij_Perez-Aros_2022}). This together with the further assumptions entail that $\scr{M}(x,v)$ is a singleton for $\mu_{\zeta}$ almost all $v \in \sph$. As a result, the right-hand side of \eqref{eq:main:theorem} is a singleton, exactly that of \eqref{eq:nablaphi}. Given that $\partial \phi(x)$ is not empty we must in fact have equality in \eqref{eq:main:theorem}. Since $\phi$ is moreover locally Lipschitzian near $x$, we may invoke Proposition 2.2.4 in \cite{Clarke_1987} to arrive at the result.
\end{proof}
\begin{example}[Example \ref{Example41} revisited]
    Let us notice that due to \cite[Theorem 1.17]{Mordukhovich_2006} the normal cone to the set $\mathcal{S}(x)$ at any $z\in \bd(\mathcal{S}(x))$ (for $x$ close enough to $\bar{x}$) is given by 
    \begin{align*}
        \nmord_{\mathcal{S}(x)}(z)=\left\{  \lambda  \begin{pmatrix}  
         z_2 + 2\\ z_1 +2
        \end{pmatrix}: \lambda  \leq 0 \right\}.
    \end{align*}
    and thanks to \eqref{computation_coderivative_example} the coderivative $D^\ast \S(x,z)( z^\ast) $ reduces to one single element. Hence, by Corollary  \ref{Corollary:theo:Main} we have that the probabability function $\varphi$ given in Example \ref{Example41} is continuously differentiable around $\bar{x}$.
\end{example}

\subsection{Application to set-valued maps presented in the form of an inclusion}
In this section we present an application for the establishment of Lipschitz continuity and subgradient formul\ae\/ of probability functions generated by concrete constrained systems. In order to simplify the analysis we will assume that $\xi$ has continuous density $f_\xi$.  

Let us consider the probability function 
		\begin{equation}\label{probfunc2}
			\varphi(x):=\mathbb{P}\left( \omega \in \Omega : \Psi_i(x,\xi(\omega))
\in C_i \,,\forall i=1,\ldots,s \right) 
		\end{equation} 
where $\Psi_i:\mathbb{R}^n \times \mathbb{R}^m \to \mathbb{R}^q$ are continuously differentiable mappings and $C_i$ are closed sets such that $\{z:\Psi_i(\bar{x},z)\in C_i\}$ is convex. We can observe here that $C_i$ could for instance be a convex cone, such as the cone of positive definite matrices. The just given setting thus allows us to extend the classic analysis from e.g., \cite{Hantoute_Henrion_Perez-Aros_2017, vanAckooij_Henrion_2016} involving inequalities, to the situation of conic programming. 

We define the finite and infinite directions with respect to $\{\Psi_i,C_i\}$ as the sets defined by
	 \begin{eqnarray*}
	 	F_{\Psi_i}(\bar{x}) &:=\{v\in \mathbb{S}^{m-1}\ |\ \exists r\geq 0: \Psi_i(x,rLv)\in \bd C_i \}, \\
	I_{\Psi_i}(\bar{x}) &:=\{v\in \mathbb{S}^{m-1}\ |\ \forall r\geq 0: \Psi_i(x,rLv)\in \inte C_i\},
	 \end{eqnarray*}%
	 respectively. We also consider
  \begin{equation*}
      F_{\Psi}(\bar{x}):=\{v\in \mathbb{S}^{m-1}\ |\ \exists r\geq 0: \Psi_i(x,rLv)\in \bd C_i \text{ for some }i=1,\ldots,s\},
  \end{equation*}
  and the radial functions
  \begin{equation*}
      \rho_{\Psi_i}(x,v):=\sup\{r>0:\Psi_i(x,rLv)\in C_i\},
  \end{equation*}
  and 
  \begin{equation*}
      \rho_{\Psi}(x,v):=\min_{1\leq i\leq s}\rho_{\Psi_i}(x,v).
  \end{equation*}

	\begin{definition}[$\eta$-growth condition for generalized constraints systems]
	 	\label{def:growth:sujerctive_gradient} 
	 	Let $\eta: \Re
	 	\to  [0,+\infty)$ be as in Definition \ref{def:growth:setvalued}.
%
	We say that the family $\{\Psi_i,C_i\}_{i=1}^s$ satisfies
	the $\eta$-growth condition at $\bar{x}$ if
	 	for some $l >0$
	 	\begin{equation}\label{growth_condition:surjective_gradient}
   \nabla_z\Psi_i(x,z) \text{ is surjective and }
	  \| \nabla_x \Psi_i(x,z)\| \leq l\eta(\|z\|)  \kappa_i(x,z) 
	 	,  \;
	 		\forall x\in \mathbb{B}_{1/l} (\bar{x}), \;\; \forall  z\in \{z':\Psi_i(x,z')\in \bd C_i\},
	 	\end{equation}
	 	for all $i=1,\ldots,s$ and where $\kappa_i(x,z)=\inf\{\| \nabla_z \Psi_i(x,z)^\ast y^\ast\|:\|y^\ast\|=1\}$.
	 \end{definition}

	\begin{corollary}\label{Cor_const_systems}
	Let us suppose that  $\xi$ has continuous density function. Suppose that a point of interest $\bar{x}$  such that the family $\{\Psi_i,C_i\}$ 
		satisfies the $\eta$-growth
		condition given in Definition \ref{def:growth:sujerctive_gradient} at $\bar{x}$ and  $ \Psi_i(\bar   x,0)\in \inte (C_i)$ for all $i=1, \ldots,s$.  Moreover, suppose that   there exists a neighbourhood $U$ of $\bar{x}$  such that $\inte\{z : \Psi_i(x,z)\in C_i  \}=\{z:\Psi_i(x,z)\in \inte C_i\}$ for all $i=1,\ldots,s$  and all $x\in U$. Then the probability function \eqref{probfunc2} is locally Lipschitz around $x\in U'$ and
 \begin{align}
 \label{pmord_varphi_cor_main}
     \pmord \varphi(x)\subseteq \int\limits_{F_{\Psi}(x)}  	\bigcup _{i\in T_x(v)} 
 	    \left\{\frac{ \theta(\rho(x,v),v)  \nabla_x \Psi_i(x, \rho( x, v)Lv )^\ast y^\ast  }{  \langle \nabla_z \Psi_i(x, \rho( x, v)Lv  )^\ast y^\ast, Lv \rangle    }:
      \begin{array}{c}
        y^\ast\in \nmord_{C_i}(\Psi_i(x,\rho( x, v)Lv))\\
       \text{ and } \| y^\ast\|=1
      \end{array}
      \right\}  d\mu_\zeta(v) , 
 \end{align}
 for all  $x \in U'$ where $T_x(v)=\{i\in \{1,\ldots,s\}:\rho_{\Psi_i}(x,v)=\rho_{\Psi}(x,v)\}$.
 
\label{cor:appsystem}		
	\end{corollary}
	
\begin{proof}
Let us consider the set-valued mapping $\S_i(x):= \{z : \Psi_i(x,z)\in C_i  \}$. Let us check that $\S_i$ satisfies the assumptions   \eqref{Assump_Setvalued}. Indeed, it has convex values. Using  the   continuity of $\Psi_i$,  closedness of $C_i$ and  the fact that $0\in\inte \{z:\Psi_i(\bar{x},z)\in C_i\}$ we can conclude that $\S_i$ satisfies Items a) and c) in Assumption \eqref{Assump_Setvalued}.  Let us prove that $\S_i$ is locally Lipschitz-like around $(x,z)\in\gph \S_i$ with $x\in U$. Since, $\inte \S_i(x)=\{z:\Psi_i(x,z)\in \inte C_i\}$ for all $x\in U$  (shrinking $U$ if necessary), when $z\in\inte \S_i(x)$ we have that $(x,z)\in \inte(\gph \S_i)$ by continuity of $\Psi_i$ and in consequence $D^\ast \S_i(x,z)(z^\ast)$ is either $0\in \Re^n$ or empty for any $z^\ast\in \Re^m$. Moreover, when $(x,z)\notin \gph \S_i$,  $D^\ast \S_i(x,z)(z^\ast)$ is empty for any $z^\ast\in \Re^m$. Therefore, it is enough to prove the Lipschitz-like property around $(x,z)$ such that $x\in U$ and $z\in \bd \S_i(x)=\{z:\Psi_i(x,z)\in \bd C_i\}$. This follows from \cite[Corollary 4.37 (i)]{Mordukhovich_2006} upon noting that the conditions therein are satisfied by the surjectivity of $\nabla_z \Psi_i(x,z)$ given by \eqref{growth_condition:surjective_gradient}. Now, let us prove that $\S_i$ satisfies the $\eta$-growth
		condition for set-valued mappings at $\bar{x}$ given in Definition \ref{def:growth:setvalued}.   
  Similarly as before, it is enough to prove that there exists $l>0$ such that \eqref{thetagrowth2:setvalued} is satisfied for $x\in \mathbb{B}_{1/l}(\bar{x})$ and $z\in \bd \S_i(x)$. Since $\{\Psi_i,C_i\}$ satisfy the growth
	condition in Definition \ref{def:growth:sujerctive_gradient} at $\bar{x}$, we have that there exists $l>0$ such that $\{\Psi_i,C_i\}$ satisfy \eqref{growth_condition:surjective_gradient}. Taking $l >0$ large enough if necessary we can assume that $\mathbb{B}_{1/l}(\bar{x}) \subseteq U$. Hence, by \cite[Theorem 4.31]{Mordukhovich_2006}  we get that  for $x\in \mathbb{B}_{1/l}(\bar{x})$ and $z\in \bd \S_i(x)=\{z:\Psi_i(x,z)\in \bd C_i\}$ the following computation holds
		\begin{equation*}
 	    D^\ast\S_i(x,z)(-z^\ast)=\{x^\ast: (x^\ast,z^\ast)\in \nabla \Psi_i(x,z)^\ast \nmord_{C_i}(\Psi_i(x,z)) \}.
 	\end{equation*}
  Furthermore, by assumption on $\nabla_z \Psi_i(x,z)$ we have that 
  \begin{equation*}
      \|\nabla_z \Psi_i(x,z)^\ast y^\ast\|\geq \kappa_i(x,z) \|y^\ast\| \text{ for all } y^\ast\in \nmord_{C_i}(\Psi_i(x,z))
  \end{equation*}
  where $\kappa_i(x,z) \in (0,+\infty)$ as a result of \cite[Lemma 1.18]{Mordukhovich_2006}. Then
  \begin{equation*}
      \|D^\ast\S_i(x,z)\|\leq \max\{\| \nabla_x \Psi_i(x,z)^\ast y^\ast\|: \|y^\ast\|\leq \kappa_i(x,z)^{-1}\}\leq \kappa_i(x,z)^{-1}\|\nabla_x \Psi_i(x,z)\|
  \end{equation*}
 	and hence the $\eta$-growth
		condition for set-valued mappings follows by considering the same non-decreasing mapping $\eta$.  Therefore, since the family $\S_i$ satisfies the assumptions of Theorem \ref{theo:Main}, we have that 
 	 $\varphi$ is locally Lipschitz around $x\in U$.
 	Now let us verify \eqref{pmord_varphi_cor_main}. We have that $$\nge_{\S_i(x)}(\rho(x,v)Lv)=\{\nabla_z \Psi_i(x,\rho(x,v)Lv)^\ast y^\ast:y^\ast\in \nge_{C_i}(\Psi_i(x,\rho(x,v)Lv))\}$$ and
 	\begin{equation*}
 	    D^\ast\S_i(x,\rho(x,v)Lv)(-z^\ast)=\{x^\ast: (x^\ast,z^\ast)\in \nabla \Psi_i(x,\rho(x,v)Lv)^\ast \nmord_{C_i}(\Psi_i(x,\rho(x,v)Lv))\}.
 	\end{equation*}
 	Clearly $F_{\Psi_i}(x)=\mathcal{F}_i(x)$ entailing  $I_{\Psi_i}(x)=\mathcal{I}_i(x)$. And as was stated in \cite{vanAckooij_Perez-Aros_2022} when $f_\xi$ is continuous,  $\bar{\theta}=\theta$ and $I_{\theta}(\rho(x,v),v)=\{\theta(\rho(x,v),v)\}$. Hence \eqref{pmord_varphi_cor_main} follows from \eqref{eq:main:theorem}
 	where we omitted the closure operator due to $\H=\Re^n$.
\end{proof}	 
\begin{example}
    Let us consider   closed convex sets $C_i\subseteq \mathbb{R}^q$ with nonempty interior, and the separate variable functions $\Psi_i(x,z):= A_i(x) + B_iz$, where $A_i : \mathbb{R}^n \to \mathbb{R}^q$ is  a continuously differentiable mapping and $B_i \in \mathbb{R}^{q \times m} $ is a surjective matrix. Define the probability function 
    \begin{align}\label{Proba_example43}   
        \varphi(x):= \mathbb{P}\left(  \Psi_i(x,\xi) \in C_i \text{ for } i=1,\ldots, s\right),
    \end{align}
    where $\xi$ is a random vector with (continuous) density $f_\xi$ satisfying $\lim_{ \|z\| \to \infty} \|z\|^m f(z)=0$. The \emph{open mapping theorem} implies that \begin{align*}
     \inte\{z : \Psi_i(x,z)\in C_i  \}=\{z:\Psi_i(x,z)\in \inte C_i\} \text{ for all }x\in \mathbb{R}^n \text{ and all } i=1, \ldots,s.
    \end{align*} 
    Furthermore, let us consider  a point $\bar{x}$ such that $A_i(\bar{x}) \in \inte(C_i)$ for all $i=1,\ldots, s$. Hence, by the linearity of $B_i$ and convexity of $C_i$ it is easy to check that $\{z:\Psi_i(\bar{x},z)\in C_i\}$ is convex. Moreover, in this particular example the functions $\kappa_i$ are strictly positive constants  (see  \cite[Lemma 1.18]{Mordukhovich_2006}), so the $\eta$-growth condition for generalized constraints systems given in Definition \ref{def:growth:sujerctive_gradient} holds at $\bar{x}$ by considering $\eta(t)=1$ and $l>0$ large enough. Therefore, by Corollary \ref{Cor_const_systems} the probability function $\varphi$ given in \eqref{Proba_example43} is locally Lipschitz around $\bar{x}$.
\end{example}
\subsection{Application to set-valued maps given by quasi-convex inequalities}

Now, for $i=1,...,s$, we assume $g_i:\mathbb{R}^n \times \mathbb{R}^m \to \mathbb{R}$ being continuously differentiable and quasi-convex in $z$. Consider the probability function 
		\begin{equation}\label{probfunc3}
			\varphi(x):=\mathbb{P}\left( \omega \in \Omega : g_i(x,\xi(\omega))
			\leq 0,\,\forall i=1,\ldots,s \right).
		\end{equation}
Associated with a point of interest $\bar{x}$ such that $g_i(\bar{x},0)<0$, we define the finite and infinite directions with respect to $g$ as the sets defined by
	 \begin{eqnarray*}
	 	F_{g_i}(\bar{x}) &:=\{v\in \mathbb{S}^{m-1}\ |\ \exists r\geq 0:g_i(\bar{x},rLv)=0\}, \\
	I_{g_i}(\bar{x}) &:=\{v\in \mathbb{S}^{m-1}\ |\ \forall r\geq 0:g_i(\bar{x},rLv)< 0\},
	 \end{eqnarray*}%
	 respectively. We also consider
  \begin{equation*}
      F_{g}(\bar{x}):=\{v\in \mathbb{S}^{m-1}\ |\ \exists r\geq 0: g(\bar{x},rLv)=0\},
  \end{equation*}
  where $g(x,z)=\max\limits_{1\leq i\leq s} g_i(x,z)$ and the radial functions
  \begin{equation*}
      \rho_{g_i}(x,v):=\sup\{r>0:g_i(x,rLv)\leq 0\},
  \end{equation*}
  and 
  \begin{equation*}
      \rho_{g}(x,v):=\min_{1\leq i\leq s}\rho_{g_i}(x,v).
  \end{equation*}

\begin{definition}[$\eta$-growth condition for inequality systems]
	 	\label{def:growth:g_i:diff} 
	 	Let $\eta: \Re
	 	\to  [0,+\infty)$ be as in Definition \ref{def:growth:setvalued}.
%
	We say that the family of continuously differentiable  mappings $\{g_i\}_{i=1}^s$ satisfies
	the $\eta$-growth condition at $\bar{x}$ if
	 	for some $l >0$
	 	\begin{equation*}
   \nabla_zg_i(x,z)\not=0 \text{ and }
	 	\| \nabla_x g_i(x,z) \|\leq l\eta(\|z\|)\|\nabla_z g_i(x,z)\|
	 	,  \;
	 		\forall x\in \mathbb{B}_{1/l} (\bar{x}), \;\; \forall  z\in \{ z' : g_i(x,z')=0\}
	 	\end{equation*}
	 	for all $i=1,\ldots,s$.
	 \end{definition}

	\begin{corollary}\label{cor_main2}
	Let us suppose that  $\xi$ has continuous density function. Suppose that a point of interest $\bar{x}$ is such that   $g_i(\bar{x},0) <0$ for all $i=1,\ldots,s$ and that the family $g_i$ 
		satisfies the $\eta$-growth
		condition given in Definition \ref{def:growth:g_i:diff}  at $\bar{x}$.
		Assume, moreover, that  $\inte\{z\in\Re^m:g_i(x,z)\leq 0\}=\{z\in\Re^m:g_i(x,z)<0\}$ for all $i=1,\ldots,s$ and all $x$ in some neighbourhood of $\bar{x}$.
 Then the probability function \eqref{probfunc3} is locally Lipschitz around $x\in U'$ and
 \begin{align}
 \label{pmord_varphi_cor_main2}
     \pmord \varphi(x)\subseteq \int\limits_{ F_g(x)}  	\bigcup _{i\in T_x(v)} 
 	    \frac{ \theta(\rho_g(x,v),v)  \nabla_x g_i(x, \rho_g( x, v)Lv )  }{  \langle \nabla_z g_i(x, \rho_g( x, v)Lv  ), Lv \rangle    }  d\mu_\zeta(v) , \text{ for all } x \in U'
 \end{align}
 where $T_x(v)=\{i\in \{1,\ldots,s\}:\rho_{g_i}(x,v)=\rho_{g}(x,v)\}$. Furthermore, if $\# T_x(v)=1 $, a.s. $v\in F_g(x)$, for all $x\in U'$, then the probability function \eqref{probfunc3} is  continuously differentiable  for all $x\in U'$ and
		\begin{align}
		\label{gradient_varphi_cor_main}
			\nabla \varphi (x)= -\int\limits_{F_g(\bar x)} \theta(\rho(x,v),v) \frac{  \nabla_x g_{T_x(v)}(x, \rho_g( x, v)Lv  )  }{  \langle \nabla_z g_{T_x(v)}(x, \rho_g ( x, v)Lv ), Lv \rangle    }  	 d\mu_\zeta(v), \text{ for all } x \in U'.
		\end{align}
\label{cor:appsystem2}		
	\end{corollary}
	
\begin{proof}
It Follows from Corollary \ref{cor:appsystem} with $g_i=\Psi_i$ and $C_i=(-\infty,0]$. Here, it is easy to see that $\kappa_i(x,z)=\|\nabla_zg_i(x,z)\|$ and $N_{C_i}(g_i(x,z))=\{\lambda\geq0:\lambda g_i(x,z)=0\}$. Clearly $I_{g_i}(x)=I_{\Psi_i}(x)$ entailing  $F_{g_i}(x)=F_{\Psi_i}(x)$. Hence \eqref{pmord_varphi_cor_main2} follows from \eqref{pmord_varphi_cor_main}. Furthermore, when $\# T_x(v)=1$ a.s.  $v\in F_g(x)$, then the subdifferential of $\varphi$ reduces to a singleton, so \eqref{gradient_varphi_cor_main} follows due  to   \cite[Theorem 4.17]{Mordukhovich_2018}.
\end{proof}

\begin{example}
	Consider the separate variable function $g:\Re^m\times\Re^m\to\Re$ given by $g(x,z)=f(x)+\frac{1}{2}\ln(1+|c^\top z|^2)$ and $\xi\sim \mathcal{N}(0,R)$, where $h$ is a continuously differentiable function and $c\in \mathbb{R}^m\backslash\{0\}$. Hence,   the function $g$ is continuously differentiable and quasi-convex but not convex in $z$. Moreover, at a point $\bar{x}$ where $f(\bar{x}) <0$. We have that $g(\bar{x},0)<0$ and $\nabla_z g(x,z)=  \frac{c^\top z}{ 1+|c^\top z|^2 }c $, then
	\begin{align*}
	    \|\nabla_z g(x,z)\|  \geq  \left(  {e^{-2f(x)}- 1} \right)^{1/2} \| c\| \text{  for all } z \in  \{ z' : g(x,z')=0\} \text{ and } x \in  f^{-1}(-\infty,f(\bar{x})/2),
	    	\end{align*} 
	    	The above shows that the   $\eta$-growth condition for inequality systems  given in Definition \eqref{def:growth:g_i:diff} holds at $\bar{x}$. Furthermore, it is straightforward that 	  $\inte\{z\in\Re^m:g_i(x,z)\leq 0\}=\{z\in\Re^m:g_i(x,z)<0\}$ for all $x$ closed enough to $\bar{x}$. Consequently, by Corollary \ref{cor_main2} we have that the probability function 
	    	$	\varphi(x):=\mathbb{P}[  g (x,\xi )
			\leq 0 \}] $ is continuously differentiable at $\bar{x}$.
 
\end{example} 

\section{Application to Energy management}
\label{sec:applications}

Optimization problems involving probabilistic constraints have been popular in energy management problems. It is indeed so, since the latter constraints offer an intuitive representation of uncertainty. Depending on the context, various resolution strategies have been investigated, most of which not based on non-linear programming approaches trying to deal with the probability function ``as is". The classic non-linear programming approach has also been investigated for these problems and found useful. It is typically the case that very precise solutions can be found with it. Here precise means - feasible solutions offering a good balance between computational cost and conservativeness. Depending on the nature of possibly competing approximations, a difficult-to-find balance has to be found between computational cost and obtaining a feasible solution. This is especially true when the random vector is replaced with a discrete sample and the overarching problem with a combinatorial version of it, i.e., selecting the subsets of samples on which to enforce the constraints. Here, typically for small sample sizes, depending on the overall underlying structure fast computations can be achieved, but at the cost of yielding a typically unfeasible solution. In contrast, in order to achieve a ``feasible solution", a very large sample size is required. This increases computational time significantly, thus allowing for much room for approaches alternative to sampling to be competitive. This is particularly the case for the classic non-linear programming approach. Although here it is not our intention to provide a numerical assessment of alternatives, this discussion does show the interest for examining the non-linear programming approach. 

We will now present a possible structure arising in energy management that fits the overall scheme presented in this paper. To this end we will consider a set of time instants $t=1,...,T$ and energy system involving a (or several) classic generators generating $p^g_t$ (MW) at time $t$ as well as a wind turbine generating $p^w_t$ (MW). The system also has a customer load (MW) that needs to be satisfied at each time step. We will make the assumption that a wind turbine has a cubic dependency on underlying wind speed, see e.g., \cite{Bremer_Henrion_Moller_2015} and that it is reasonable to represent the latter using a Gaussian random vector. Similarly, customer load can be represented using a Gaussian random vector as well. We can thus assume that we are given $\xi \in \Re^{2T}$, where $\xi = (\xi_1,\xi_2)$ and $\xi_1 \in \Re^T$ represents wind speed, whereas $\xi_2 \in \Re^T$ represents customer load. It is particularly worthwhile to highlight that using this structure allows us to naturally represent the dependency between wind and load - recalling that this is meaningful since both depend on the underlying climate as load depends partially on temperature. We will now assume that the amount of power that the turbine generates has to be allocated to the satisfaction of load prior to actually observing the amount generated. Any excess generation will be assumed to be curtailed. Likewise, the amount of power generated through the conventional generator(s) has to be decided upon prior to observing load. The underlying inequality system will be of the form:
\begin{align}
    p^w_t & \leq c (z^1_t)^3, t=1,...,T \notag \\
    p^w_t &+ p^g_t \geq z^2_t, t=1,...,T \label{eq:thesystem} \\
    0 &\leq z^1_t,  t=1,...,T. \notag
\end{align}
Upon observing that the term $(z^1_t)^3$ can likewise be written as $(z^1_t)^4 / z^1_t$, one can equally write $p^w_t + \frac{c (z^1_t)^4}{-z^1_t} \leq 0$ as the first inequality. It is so that the given fraction is, as function of $z$, quasi-convex, resulting from the division of a convex map by a concave one. The last constraint, stating that wind-speed should be positive, ensures that the system is well defined on the relevant set of variables. The conventional generator can be assumed to have a proportional cost of generation $c^g_t$ in \euro $/ MW$. As a result we would be facing an optimization problem of the form:
\begin{align*}
    \min_{p^w,p^g} \quad & \sum_{t=1}^T c^g_t p^g_t \\
    \mbox{s.t.}\quad & \prb[ \eqref{eq:thesystem} \mbox{ holds with } \xi \mbox{ as } z] \geq p, \\
    & p^w \in P^W, p^g \in P^g,
\end{align*}
where $P^w$, $P^g$ represent classic constraints such as physical limitations on generation induced by the installed capacity. With the help of Corollary \ref{cor:appsystem2} a first order optimization approach for the latter problem can be set up. A particularly favorable approach is the recently developed DCWC algorithm \cite{Syrtseva_Oliveira_Demassy_vanAckooij_2023}. However, due to the non-convex nature of the just provided optimization problem, at best, a local solution, in fact a weaker so called critical point can be expected. This is the nature of the situation. 

We have set up a concrete instance of the given optimization problem. To this end we have taken $T=4$, $c=0.032$ (see \cite{Bremer_Henrion_Moller_2015}), $\mu = (\mu_1,\mu_2)$, with each component of $\mu_1$ having the value $4.23$ and each component of $\mu_2$ having the value $10$. The ``wind'' random variable is supposed to be auto-correlated with coefficient $0.96$ and the ``load" one with coefficient $0.8$. Wind and load themselves are negatively correlated with coefficient $-0.3$. The total $2T \times 2T$ correlation matrix thus has a specific $2 \times 2$ block structure where the $(1,2)$ block is the pointwise product of appropriate elements in the $(1,1)$ and $(2,2)$ block and $-0.3$. The $(2,1)$ block is set up symmetrically. The variances $v = (v_1, v_2)$ consists of two constant vectors, the first with elements $1.54$, the second the all-one vector. The cost $c^g$ is set as follows: $c^g = 5$. Finally the decision vector has a lower and upper bound, $0 \leq p^w_t \leq 8$, $0 \leq p^g_t \leq 20$. The probability level was taken to be $p=0.8$. 

The DCWC algorithm manages to solve the problem finding a high quality numerically almost feasible solution after about 130 iterations. The total number of iterations to reach convergence is 907. The found solution is found to exactly have a probability value of 0.8.

\section{Conclusions}

In this paper we have given gradient formul{\ae}   of probability functions of which one of the main ingredients is a set-valued application. We have managed to establish under reasonable assumptions generalized differentiability of the latter probability functions with respect to the external parameter playing the role of the argument of the set-valued maps. The given results allow us to extend previous works investigating first order properties of probability functions. Moreover we have put the resulting formul{\ae} to use within the context of a given algorithm. This demonstrates the practicality as well.

	 \appendix
	 
	 \section{Spherical radial decomposition for nonsmooth inequality constraints}
	 
	 The intention of this section is to provide an extension of \cite[Theorem 3.1]{vanAckooij_Perez-Aros_2022} to the infinite dimensional nonsmooth setting. 
	Given a family of functions $g_i \fc{\H
	 	\times \Re^m}{\Re}$ with $i=1,\ldots,s$, the probability function that  we consider here is
	 \begin{equation}\label{probfunc2:Appendix}
	 	\varphi(x):=\mathbb{P}[\{ \omega \in \Omega : g_i(x,\xi(\omega))
	 	\leq 0 \,\forall i=1,\ldots,s\}], 
	 \end{equation}
	 where $\xi $ is a $m$-dimensional random vector admitting a
	 density called $f_{\xi}$ with respect to the Lebesgue measure $\lambda_m$. Let us recall that $f_{\xi}$ is assumed bounded on
	 compact sets (see \eqref{eq:boundedassump:set}).
 
	 Using the spherical radial decomposition  we get that 
	 \begin{equation*}
	 	\varphi(x)= \int\limits_{\mathbb{S}^{m-1}} e(x,v)  d
	 	\mu_\zeta(v),
	 \end{equation*}
	 where $e: \H \times \mathbb{S}^{m-1} \to \mathbb{R}\cup
	 \{ \infty\}$ is the \emph{radial probability-like} function given
	 by
	 \begin{align*}
	 	e(x,v) = \displaystyle  \int\limits_{  \{ r\geq 0 :
	 		g(x, rLv)     \leq 0   \}      }   \theta(r,v)
	 	dr, \,\text{ where }
	 	g(x,z):=\max\{g_i(x,z):\,i=1,\ldots,s\},
	 \end{align*}
	 and $\theta$ is the \emph{density-like} function  defined in  \eqref{transformationftorho}.
	 Given a point of interest $\bar x$ we assume that on a neighborhood $U$ of $\bar{x}$ we have that $g_i$ is locally Lipschitz around any $(x,z)\in U\times \Re^m$, convex in the second variable for all $i=1,\ldots,s,$ and that
	  \begin{equation}\label{nonemptyinterior_Ass:Appendix}
	 	g_i(x,0) < 0, \quad \text{ for all } x\in U \text{ and }i=1,\ldots,s.
	 \end{equation}
	 For $x\in U$ we define the sets of
	 finite and infinite directions with respect to $g_i$ as
	 \begin{eqnarray}
	 	F_{g_i}(x) &:=\{v\in \mathbb{S}^{m-1}\ |\ \exists r\geq 0:g_i(x,rLv)=0\}, \label{appendix:finite_direction_sets} \\
	I_{g_i}(x) &:=\{v\in \mathbb{S}^{m-1}\ |\ \forall r\geq 0:g_i(x,rLv)< 0\}\label{appendix:infinite_direction_sets},
	 \end{eqnarray}%
	respectively. The finite and infinite directions with respect to $g$ can be defined analogously and in fact they identify with $F_g(x)=\cup_{i=1}^s F_{g_i}(x)$ and $I_g(x)=\cap_{i=1}^s I_{g_i}(x)$.
	
	Now, let us define the radial functions 
	\begin{align}
	\label{appendix:radial_function_i}
		\rho_{g_i} \left( x,v\right) &:=\sup\left\{  r>0 : g_i(x,rLv)\leq 0 \right\} \text{ for all } i=1,\ldots,s,
	\end{align}
	and
	\begin{align}
	\label{appendix:radial_function}
	    \rho_g(x,v)=\min\{\rho_{g_i}(x,v):i=1,\ldots,s\}.
	\end{align}
	
	Finally, let us denote the active set at $(x,v)$ with $\rho_g(x,v)<+\infty$ as
	\begin{equation*}
	    T_x(v):=\{i=1,\ldots,s: \rho_{g_i}(x,v)=\rho_{g}(x,v)\}.
	\end{equation*}
	 The following lemma corresponds to a generalization of \cite[Lemma 3.4]{vanAckooij_Perez-Aros_2022}.
	 
	  \begin{lemma}\label{estimate_of_pmord_rho:Appendix}
	  Let $x\in U$ be given, then we have that
	 \begin{equation}
	 \label{appendix:pmord:rho_i}
	     \pmord_x \rho_{g_i}(x,v)\subseteq -\cl{}\conv\left\{ \frac{1}{\langle z^\ast,Lv\rangle} x^\ast: (x^\ast,z^\ast)\in \pmord g_i(x,\rho_{g_i}(x,v)Lv)\right\} \text{ for all } i=1, \ldots, s \text{ and }  v\in F_{g_i}(x).
	 \end{equation}
	 Moreover, 
	 \begin{equation}\label{appendix:pmord:rho}
	     \pmord_x \rho_{g}(x,v)\subseteq -\cl{}\conv\left\{ \frac{1}{\langle z^\ast,Lv\rangle} x^\ast:
	     \begin{array}{c}
	     (x^\ast,z^\ast)\in \pmord g_i(x,\rho_g{}(x,v)Lv)\\
	   \text{ and }i\in T_x(v)\end{array}\right\}.
	 \end{equation}
	 \end{lemma}
	 
	 \begin{proof}
	 Fix $\bar{x}\in U$, $i\in \{1,\ldots,s\}$ and $\bar{v}\in F_{g_i}(\bar{x})$. To obtain \eqref{appendix:pmord:rho_i} let us first prove that for every $y^\ast\in \pmord_x\rho_{g_i}(\bar{x},\bar{v})$ and every $w\in \H$, there exists $(x^\ast,z^\ast)\in \pmord g_i(\bar{x},\rho_{g_i}(\bar{x},\bar{v})L\bar{v})$ such that
	 \begin{equation*}
	     \langle y^\ast,w\rangle\leq \frac{-1}{\langle z^\ast,L\bar{v}\rangle}\langle x^\ast,w\rangle.
	 \end{equation*}
	 
	 By continuity of $\rho_{g_i}$ (see \cite[Lemma 1]{Hantoute_Henrion_Perez-Aros_2017}) and since $g_i$ are locally Lipschitz, there exists $\epsilon'>0$ such that for every $x\in \mathbb{B}(\bar{x},\epsilon')$ and $z\in \mathbb{B}(\rho_{g_i}(\bar{x},\bar{v})L\bar{v},\epsilon')$ we have that
	 \begin{equation}\label{constants1:appendix}
	 \rho_{g_i}(x,\bar{v})\leq M
	 \text{, }
	 g_i(x,0)\leq -\gamma
	 \end{equation}
	 and
	 \begin{equation}\label{constants2:appendix}
	 \pmord g_i(x,z)\subseteq r\mathbb{B}^\ast
	 \end{equation}
	 for some constants $\gamma,M,r>0$.

    We claim that for every $y^\ast\in \pfrech_x\rho_{g_i}(x,\bar{v})$ with $x\in \mathbb{B}(\bar{x},\epsilon'/2)$ and every $w\in \H$  there exists $(x^\ast,z^\ast)\in \pmord  g_i(x,\rho_{g_i}(x,\bar{v})L\bar{v})$ such that
	  \begin{equation*}
	      \langle y^\ast,w\rangle\leq \frac{-1}{\langle z^\ast,L\bar{v}\rangle}\langle x^\ast,w\rangle.
	  \end{equation*}
	  To see this, let $w\in \H$ and consider $t_k\to 0^+$ such that 
	  \begin{equation*}
	      x+t_kw\in \mathbb{B}(\bar{x},\epsilon'/2) \,\text{ and }\, \rho_{g_i}(x+t_kw,\bar{v})L\bar{v}\in \mathbb{B}(\rho_{g_i}(x,\bar{v})L\bar{v},\epsilon'/2), \text{ for all }k.
	  \end{equation*}
	  Applying the mean value inequality in \cite[Corollary 3.51]{Mordukhovich_2006} we get
	  \begin{equation*}
	   g_i(x+t_kw,\rho_{g_i}(x+t_kw,\bar{v})L\bar{v})-g_i(x,\rho_{g_i}(x,\bar{v})L\bar{v})\leq -t_k\langle x_k^\ast,w\rangle+
	      [\rho_{g_i}(x,\bar{v})-\rho_{g_i}(x+t_kw,\bar{v})]\langle z_k^\ast,L\bar{v}\rangle,
	  \end{equation*}
	  for some 
	  \begin{equation*}
	      (x_k^\ast,z_k^\ast)\in \pmord g_i(x_k,z_k)
	  \end{equation*}
	   with
	   \begin{equation*}
	      x_k\in [x+t_kw,x)\text{ and }z_k\in \left[\rho_{g_i}(x+t_kw,\bar{v})L\bar{v},\rho_{g_i}(x,\bar{v})L\bar{v}\right).
	   \end{equation*}
	  Hence, taking into account, from the definition of $\rho_i$, that
	  \begin{equation*}
	     g_i(x,\rho_{g_i}(x,\bar{v})L\bar{v})=0\,
	     \text{ and }\,
	     g_i(x+t_kw,\rho_{g_i}(x+t_kw,\bar{v})L\bar{v})=0,
	  \end{equation*}
	  it follows that
	  \begin{equation*}
	    [\rho_{g_i}(x+t_kw,\bar{v})-\rho_{g_i}(x,\bar{v})]\langle z_k^\ast,L\bar{v}\rangle\leq -t_k\langle x_k^\ast,w\rangle.
	  \end{equation*}
	  Now, by \cite[Lemma 4.8]{https://doi.org/10.48550/arxiv.2112.05571}, $z_k^\ast \in \partial_z g_i(x_k,\rho_{g_i}(x_k,\bar{v})L\bar{v})+\bar{z})$ and hence, by \cite[Lemma 3]{Hantoute_Henrion_Perez-Aros_2017} together with  \eqref{constants1:appendix}, we obtain that $\langle z_k^\ast, L\bar{v}\rangle>0$. Therefore,
	  \begin{equation*}
	    \frac{\rho_{g_i}(x+t_kw,\bar{v})-\rho_{g_i}(x,\bar{v})}{t_k}\leq \frac{-1}{\langle z_k^\ast,L\bar{v}\rangle}\langle x_k^\ast,w\rangle.
	  \end{equation*}
	   Considering $\epsilon_k\to0^+$ with $\epsilon_k<\epsilon'/2$ for all $k$, by definition of the basic subdifferential, we have that there exists $(\hat{x}^\ast_k,\hat{z}^\ast_k)\in \pfrech g_i(\hat{x}_k,\hat{z}_k)$ with
	   $\|\hat{x}_k-x_k\|\leq \epsilon_k$,
	   $\|\hat{z}_k-z_k\|\leq \epsilon_k$, $\|\hat{z}^*_k-z^*_k\|\leq \epsilon_k$ and such that
	   \begin{equation}\label{eq1:lemma:appendix}
\langle z_k^\ast,L\bar{v}\rangle	    \frac{\rho_{g_i}(x+t_kw,\bar{v})-\rho_{g_i}(x,\bar{v})}{t_k}\leq -(\langle \hat{x}_k^\ast,w\rangle -\epsilon_k).
	  \end{equation}
	    Now using \eqref{constants2:appendix}, we have that $\|\hat{x}^\ast_k\|\leq r$ and $\|\hat{z}^\ast_k\|\leq r$. Since $\H$ is reflexive there exists a subsequence $(\hat{x}^\ast_{n_k},\hat{z}^\ast_{n_k})$ and some $(x^*,z^*)\in\H\times\mathbb{R}^{m}$ such that $x^*_{n_k}\rightharpoonup_k x^\ast$ and $z^*_{n_k}\rightarrow z^*$. By definition of the basic sub-differential $(x^\ast,z^\ast)\in \pmord g_i(x,\rho_{g_i}(x,\bar{v})L\bar{v})$. Again by \cite[Lemma 4.8]{https://doi.org/10.48550/arxiv.2112.05571}, $z^\ast \in \partial_z g_i(x,\rho_{g_i}(x,\bar{v})L\bar{v}))$. Hence, by \cite[Lemma 3]{Hantoute_Henrion_Perez-Aros_2017} together with  \eqref{constants1:appendix}, we have that $\langle z^\ast,L\bar{v}\rangle>0$. Thus, by taking the inferior limit in \eqref{eq1:lemma:appendix}, we conclude the proof of the claim by recalling the definition of the regular subdifferential.
	  
	  Now, let $y^\ast\in \pmord _x\rho_{g_i}(\bar{x},\bar{v})$. Then there exist $y_l^\ast \to y^\ast$ and $x_l\to \bar{x}$ with $y^\ast_l\in \pfrech_x \rho_{g_i}(x_l,\bar{v})$. For $l$ large enough such that
	  \begin{equation*}
	      \|x_l-\bar{x}\|\leq \epsilon'/2\, \text{ and }\,
	      \|\rho_{g_i}(x_l,\bar{v})L\bar{v}-\rho_{g_i}(\bar{x},\bar{v})L\bar{v}\|\leq \epsilon'/2, 
	  \end{equation*}
	  we apply the claim proved above to obtain that there exists $(x_l^\ast,z_l^\ast)\in \pmord  g_i(x_l,\rho_{g_i}(x_l,\bar{v})L\bar{v})$ such that
	  \begin{equation*}
	      \langle y_l^\ast,w\rangle\leq \frac{-1}{\langle z_l^\ast,L\bar{v}\rangle}\langle x_l^\ast,w\rangle.
	  \end{equation*}
	  By definition of the basic subdifferential and considering $\epsilon_l\to0^+$ with $\epsilon_l<\epsilon'/2$ for all $l$,  there exists $(\hat{x}_l^\ast,\hat{z}_l^\ast)\in \pfrech g_i(\hat{x}_l,\hat{z}_l)$ with
	  $\|\hat{x}_l-x_l\|\leq \epsilon_l$,
	  $\|\hat{z}_l-\rho_{g_i}(x_l,\bar{v})L\bar{v}+\bar{z}\|\leq \epsilon_l$, $\|\hat{z}^*_l-z^*_l\|\leq\epsilon_l$ and such that
	  \begin{equation}\label{eq2:lemma:appendix}
	      \langle y_l^\ast,w\rangle\leq \frac{-1}{\langle z_l^\ast,L\bar{v}\rangle}(\langle \hat{x}_l^\ast,w\rangle-\epsilon_l).
	  \end{equation}
	   Again, using \eqref{constants2:appendix}   we obtain that (under subsequence) $(\hat{x}_l^\ast,\hat{z}_l^\ast)\rightharpoonup_l (x^\ast,z^\ast)\in \pmord g_i(\bar{x},\rho_{g_i}(\bar{x},\bar{v})L\bar{v})$. Therefore, letting $l\to\infty$ in \eqref{eq2:lemma:appendix}, we conclude that
	   \begin{equation*}
	     \langle y^\ast,w\rangle\leq \frac{-1}{\langle z^\ast,L\bar{v}\rangle}\langle x^\ast,w\rangle
	 \end{equation*}
	 for some $(x^\ast,z^\ast)\in \pmord g_i(\bar{x},\rho_{g_i}(\bar{x},\bar{v})L\bar{v})$. Let us notice that this last result implies that
	 \begin{equation*}
	     \langle y^\ast,w\rangle\leq \sigma_{\mathcal{A}(\bar{x},\bar{v})}(w)\text{ for all } y^\ast\in \pmord_x\rho_{g_i}(\bar{x},\bar{v}) \text{ and for all } w\in \H
	 \end{equation*}
	 where
	 \begin{equation*}
	     \mathcal{A}(\bar{x},\bar{v}):=\left\{\frac{-1}{\langle z^\ast,L\bar{v}\rangle} x^\ast : (x^\ast,z^\ast)\in \pmord g_i(\bar{x},\rho_{g_i}(\bar{x},\bar{v})L\bar{v}) \right\}.
	 \end{equation*}
	 Therefore, $\sigma_{\pmord_x\rho_{g}(\bar{x},\bar{v})}(w)\leq \sigma_{\mathcal{A}(\bar{x},\bar{v})}(w)$, for all $w\in \H$, which entails \eqref{appendix:pmord:rho_i} due to \cite[Theorem 2.4.14 (vi)]{Zalinescu_2002}. Finally, \eqref{appendix:pmord:rho} follows from \cite[Proposition 1.113]{Mordukhovich_2006}.
	 \end{proof}
	 
	 Let us recall the function $I_\theta : \Re_+ \times
	\sph \tto \Re_+$ given in \eqref{defin_Itheta} by
	\begin{equation}\label{defin_Itheta:Appendix}
		I_{\theta}(r,v) := [\underline{\theta}(r,v),\overline{\theta}_+(r,v)]
		\cup[\underline{\theta}^-(r,v),\overline{\theta}(r,v) ],
	\end{equation}
	where, $\overline{\theta}$, ${\bar{\theta}}_{+}$, $  \underline{\theta}$ and $ \underline{\theta}^{-}$ were defined in   \eqref{defin_thetafunctions}. 
	As deduced in \cite{vanAckooij_Perez-Aros_2022}, property \eqref{eq:boundedassump:set} implies that
	\begin{equation}\label{overline_theta_finite:Appendix}
	    I_\theta(r,v)\subseteq [0,\overline{\theta}(r,v)]\subseteq [0,+\infty).
	\end{equation}
Furthermore, from the definition of $\overline{\theta}(r,v)$, we have
\begin{equation}\label{overline_theta_bound:Appendix}
    \overline{\theta}(r,v)\leq M_\theta, \,\forall r\leq M,\text{ and }v\in\mathbb{S}^{m-1}
\end{equation}
where $M_\theta$ is the (finite) constant defined by
\begin{equation*}
    M_\theta:=\esssup_{(r,v)\in [0,M+1]\times \mathbb{S}^{m-1}}\overline{\theta}(r,v).
\end{equation*}
	 
	 \begin{definition}[$\eta_{\theta}$-growth condition]
	 	\label{def:thetagrowth:Appendix} Consider $ \bar{x} \in U$ and $\bar{v} \in
	 	I_g( \bar{x})$. Let $\eta_{\theta} : \Re\times \mathbb{S}^{m-1}
	 	\to  [0,+\infty]$  be a mapping such that
	 	\begin{equation*}
	 		\lim\limits_{\substack{x \to \bar{x} \\
	 				v \to \bar{v} } }  \rho_g(x,v) \bar{\theta} ( \rho_g(x,v),v) \eta_{\theta}( \rho_g(x,v),v) =0.
	 	\end{equation*}
	 	
	 	We say that the family of mappings $\{g_i\}_{i=1}^s$ satisfies
	 	the $\eta_{\theta}$-growth condition at $(\bar{x}, \bar{v})$ if
	 	for some $l >0$
	 	\begin{equation*}
	 		\| \pi_x(\pmord g_i(x,\rho_g(x,v)Lv)) \| \leq l \eta_{\theta}(\rho_g(x,v),v),  \;
	 		\forall (x,v) \in \mathbb{B}_{1/l} (\bar{x}) \times
	 		\mathbb{B}_{1/l} (\bar{v}), \;\; v\in F(x)\;\; \text{and}\;\; i\in T_x(v); 
	 	\end{equation*}
	 	where 
	 	$\| \pi_x(\pmord g_i(x,\rho_{g_i}(x,v)Lv)) \|:=\sup\{ \| x^\ast\| : (x^\ast ,z^\ast)\in \pmord g_i(x,\rho_{g_i}(x,v)Lv)$ for some $z^\ast$\}.
	 \end{definition}

	 \begin{theorem}
	 	\label{thm:clarkeappli:Appendix} 
	 	Let $\bar{x} \in \H $ be given and assume that the family of mappings $\{g_i\}_{i=1}^s$ satisfies the $\eta_{\theta}$-growth
	 	condition at $(\bar{x},v)$ for all $v\in I_g(\bar{x})$ and that  \eqref{nonemptyinterior_Ass:Appendix}
	 	holds true.	Then the probability function \eqref{probfunc2:Appendix} is locally
	 	Lipschitz around $\bar x$ and
	 	\begin{align}\label{inclusion:subdif2}
	 		\pmord \varphi (\bar{x})\subseteq -\cl{w^\ast}\left(  \int\limits_{\mathbb{S}^{ m-1}}  \mathcal{G}(\bar{x},v)    d\mu_\zeta(v) \right)
	 	\end{align}
	 	where 
	 	\begin{align*}
	 		\mathcal{G}(x,v)  &=\left\{ \begin{array}{cr}
	 		   \left\{   \frac{ \alpha  
	 		}{%
	 			\left\langle z^\ast
	 			,Lv\right\rangle } x^\ast  : \begin{array}{c}
	 	 (x^\ast,z^\ast) \in  \pmord g_i(x, \rho_g \left( x,v\right) Lv) \\ i\in T_{x}(v),\,
	 	 \alpha \in I_{\theta}(\rho_g(x,v),v) 
	 		\end{array}        \right\}, &\text{ if }    v\in F_g(x)\\ \vspace{-0.2cm}& \\
	 		\{0\},     &\text{ if }    v\in I_g(x),
	 		\end{array}     \right.
	 	\end{align*}
	 	with $I_{\theta} $   given by \eqref{defin_Itheta:Appendix}. Moreover, the closure operator can be omitted in \eqref{inclusion:subdif2}  if $\H$ is finite-dimensional.
	 \end{theorem}
	 
	 \begin{proof}
	 First  let us notice that due to Lemma \ref{estimate_of_pmord_rho:Appendix} and \cite[Proposition 3.2]{vanAckooij_Perez-Aros_2022} we have that $\pmord_x e(x,v)\subset  -\cl{}\conv \mathcal{G}(x,v)$ for all    $x\in U$ and all $v\in F_g(x)$. Moreover, similarly as in \cite[Proposition 3.4 i)]{vanAckooij_Perez-Aros_2022}, $\pfrech_x e(x,v)\subseteq \{0\}=-\cl{}\conv \mathcal{G}(x,v)$ if $v\in I_g(x)$.

	 \textbf{Claim 1:}
	 For every fixed $\bar{v}\in \mathbb{S}^{m-1}$ with $\bar{v} \in F_g(\bar{x})$, there exist neighborhoods $U_{\bar{v}}$ of $\bar{x}$ and $V_{\bar{v}}$ of $\bar{v}$ and  $K_{\bar{v}}>0$ such that
	 	\begin{equation*}
	 	\pmord_x e(x,v)\subset  -\cl{}\conv \mathcal{G}(x,v) \subseteq 
	 	K_{\bar{v}}\mathbb{B}^\ast \text{ for all }(x,v) \in U_{\bar{v}}\times V_{\bar{v}}.
	 	\end{equation*}\\
	 	By the continuity of $\rho$ (see \cite[Lemma 1]{Hantoute_Henrion_Perez-Aros_2017}) there exist neighborhoods $U_{\bar{v}}$ of $\bar{x}$, $V_{\bar{v}}$ of $\bar{v}$ and a constant $M>0$ such that $\rho_g(x,v)\leq M$ and $g(x,0)<0$ for all $(x,v)\in U_{{\bar{v}}}\times V_{\bar{v}}$. Hence, it follows that 
		 	  \begin{equation}\label{pmord:e:appendix}
	 	      	\pmord_x e(x,v) \subset  \cl{}\conv \mathcal{G}(x,v) \,\text{ for all }(x,v)\in U_{\bar{v}}\times V_{\bar{v}}.
	 	  \end{equation}
	 	   Now, for each $(x^\ast,z^\ast)\in\pmord g_i(x, \rho_g\left(x,v\right) Lv)$ in \eqref{pmord:e:appendix} we have that $z^\ast \in  \partial_z g_i(x, \rho_g \left( x,v\right) Lv)$ (see \cite[Lemma 4.8]{https://doi.org/10.48550/arxiv.2112.05571}) and in consequence (see \cite[Lemma 3]{Hantoute_Henrion_Perez-Aros_2017})
	 	   \begin{equation*}
	 	       \langle z^\ast,Lv\rangle\geq -\frac{g_i(x,0)}{\rho_g(x,v)}>0.
	 	   \end{equation*}
	 	   Hence, by \eqref{pmord:e:appendix} and \eqref{overline_theta_finite:Appendix}, for each $y^\ast\in \pmord_x e(x,v)$ there exists $(x^\ast,z^\ast)\in \pmord g_i(x, \rho_g \left( x,v\right) Lv)$ with $i\in T_x(v)$ such that
	 	   \begin{equation*}
	 	       \|y^\ast\|\leq \frac{-1}{|g_i(x,0)|}\rho_g(x,v)\overline{\theta}(\rho_g(x,v),v)\|x^\ast\|
	 	   \end{equation*}
	 	   Therefore, by \eqref{overline_theta_bound:Appendix}, the continuity of $\rho_g$ and the fact that the functions $g_i$ are locally Lipschitz, we conclude the claim.\\
	 	
\textbf{Claim 2:} For every fixed $\bar{v}\in \mathbb{S}^{m-1}$ with $\bar{v} \in I_g(\bar{x})$ and $\epsilon>0$, there exist neighborhoods $U_{\bar{v}}$ of $\bar{x}$ and $V_{\bar{v}}$ of $\bar{v}$ such that
	 	\begin{equation*}
	 	    \pfrech_x e(x,v)\subseteq -\cl{}\conv\mathcal{G}(x,v)\subseteq \epsilon\mathbb{B}^{\ast} \text{ for all }(x,v) \in U_{\bar{v}}\times V_{\bar{v}}.
	 	\end{equation*}
	 	
Let $\epsilon>0$ and consider $l>0$ given by the $\eta_\theta$-growth condition in Definition \ref{def:thetagrowth:Appendix}. By continuity of $\rho_g$ (see \cite[Lemma 1]{Hantoute_Henrion_Perez-Aros_2017}), there exist neighborhoods $U_{\bar{v}}$ of $\bar{x}$ and $V_{\bar{v}}$ of $\bar{v}$, contained in $\mathbb{B}_{1/l}(\bar{x})$ and $\mathbb{B}_{1/l}(\bar{v})$ respectively,  such that
	 	    \begin{equation*}
	 	        \rho_g(x,v)\geq l\text{ and }\rho_g(x,v)\bar{\theta}(\rho_g(x,v),v)\eta_{\theta}(\rho_g(x,v),v)\leq \epsilon' \text{ for all }(x,v)\in U_{\bar{v}}\times V_{\bar{v}}.
	 	    \end{equation*}
where $\epsilon':=(\sup_{x\in U}\frac{1}{|g(x,0)|})^{-1}\epsilon$. Now, when $v\in F_g(x)$, we may follow as in Claim 1 to obtain
	 	    \begin{equation*}
\frac{\alpha  
	 		}{%
	 			|\left\langle z^\ast
	 			,Lv\right\rangle| } \|x^\ast\| \leq \frac{1}{|g(x,0)|}\rho_g(x,v)\bar{\theta}(\rho_g(x,v),v)\eta_\theta(\rho(x,v),v)\leq \epsilon,
	 	    \end{equation*}
which proves that $\cl{}\conv\mathcal{G}(x,v)\subseteq \epsilon \mathbb{B}^\ast$ for all $(x,v) \in U_{\bar{v}}\times V_{\bar{v}}$. Finally, the first inclusion follows from the first paragraph of the proof.
	 	   
\textbf{Claim 3:} For every fixed $\bar{v}\in \mathbb{S}^{m-1}$ with $\bar{v} \in I_g(\bar{x})$ we have that
	 \begin{equation*}
	 	\pmord_x e(\bar{x},\bar{v})\subseteq \{0\}=-\cl{}\conv \mathcal{G}(\bar{x},\bar{v}).
	 \end{equation*}

Let $y^\ast\in\pmord_xe(\bar{x},\bar{v})$ and choose $x_n\to \bar{x}$ and $y_n^\ast\rightharpoonup y^\ast$ with $y_n^\ast\in \pfrech_xe(x_n,\bar{v})$. Considering a sequence $\epsilon_n\to 0^+$ and claim 2 with $\epsilon_n$ we obtain that $y^\ast=0$.\\

\textbf{Claim 4:} There exists  a neighborhood $U'$ of $\bar{x}$ such that 
\begin{equation*}
    \pfrech_x e(x,v)\subseteq \kappa\mathbb{B}^\ast \;\text{ for all } (x,v)\in U' \times\mathbb{S}^{m-1}.
\end{equation*}

Indeed, by claims 1 and 2, we have that for every fixed $\bar{v}\in \mathbb{S}^{m-1}$ there exist neighborhoods $U_{\bar{v}}$ of $\bar{x}$ and $V_{\bar{v}}$ of $\bar{v}$ and  $K_{\bar{v}}>0$ such that
	 	\begin{equation*}
	 	\pfrech_x e(x,v) \subseteq 
	 	K_{\bar{v}}\mathbb{B}^\ast \text{ for all }(x,v) \in U_{\bar{v}}\times V_{\bar{v}}.
	 	\end{equation*}

Now, since the family of neighborhoods $V_{\bar{v}}$ covers the compact set $\mathbb{S}^{m-1}$  we can pick a finite subcover, that is, there exists $N\in\mathbb{N}$ and some $v_1,\ldots,v_N\in \mathbb{S}^{m-1}$ such that
\begin{equation*}
   \mathbb{S}^{m-1}\subset \bigcup_{i=1}^N V_{v_i}.
\end{equation*}
Therefore, we choose a neighborhood $U'$ of $\bar{x}$ such that
\begin{equation*}
    U'\subset \bigcap_{i=1}^N U_{v_i}
\end{equation*}
and define $ \kappa:=\max\{K_{v_i}:\;i=1\ldots,N\}$ to conclude the claim.\\

\textbf{Claim 5:} Now by \cite[Proposition 3]{Hantoute_Henrion_Perez-Aros_2017}, it follows that the probability function \eqref{probfunc2:Appendix} is locally
	 	Lipschitz around $\bar x$ and 
	 	\begin{align*}
	 		\pmord \varphi (\bar{x})\subseteq -\cl{w^\ast}\left(  \int\limits_{\mathbb{S}^{ m-1}}  	\cl{}\conv {\mathcal{G}(\bar{x},v)}  d\mu_\zeta(v) \right).
	 	\end{align*}

From \cite[Corollary 4.4]{https://doi.org/10.48550/arxiv.1803.05521} together with claim 4 we obtain the inclusion
\begin{equation*}
    \pmord \varphi (\bar{x})\subseteq \cl{w^\ast}\left(  \int\limits_{\mathbb{S}^{ m-1}}  	\pmord_x e(\bar{x},v)    d\mu_\zeta(v) \right).
\end{equation*}
Hence, by the first paragraph of the proof together with claim 3, it follows that $\pmord_x e(\bar{x},v) \subseteq -\cl{}\conv {\mathcal{G}(\bar{x},v)}$, concluding the proof of the claim.
\\

\textbf{Claim 6:}  The inclusion \eqref{inclusion:subdif2} holds. Notice that $\mathcal{G}(x,v)$ is integrably bounded (see \cite[p. 326]{Aubin_Frankowska_2009}) by claims 1 and 2. Therefore, the claim follows as a consequence of \cite[Theorem 8.6.4]{Aubin_Frankowska_2009} since $\H$ is a separable reflexive Banach space and $\mu_\zeta$ is nonatomic.
\end{proof}
	 
We are now in measure to apply the just given Theorem \ref{thm:clarkeappli:Appendix} in the setting of section \ref{section:inner_enlargement}.

\emph{Proof of Theorem \ref{Theorem:epsilon}}
	Let us notice that the probability function \eqref{P01eps} can be written as
		\begin{align*} 
			\varphi_\epsilon(x):= \mathbb{P} (\omega \in \Omega: g(x,\xi(\omega)) \leq 0
			)
		\end{align*} 
		where $g(x,z):= \max\{g_i(x,z):i=1,\ldots,s\}$ and $ 
		    g_i(x,z)=\frac{1}{2}\dst^2(z, \S_i(x))-\frac{\epsilon^2}{2}.
	$
		Let us prove that the assumptions of Theorem \ref{thm:clarkeappli:Appendix} are satisfied. Indeed, by Assumption \eqref{Assump_Setvalued} we have that $g_i(x,0)=-\tfrac{\epsilon^2}{2}<0$ and that the ${g_i}'s$ are convex in the second variable. Due to Lemma \ref{lemma:derivative} the functions ${g_i}'s$ are locally Lipschitz around $(x,z)\in U\times \mathbb{R}^m$. Also, by Lemma \ref{domain:rho} Item $f)$, the sets of finite direction $\mathcal{F}_i$ defined in \eqref{finite_directions_Si} coincide with the sets of finite directions $F_{g_i}(x)$ defined in \eqref{appendix:finite_direction_sets}. That is, $\mathcal{F}_i(x)=F_{g_i}(x)$, and so, by taking complements we also have that $\mathcal{I}_i(x)=I_{g_i}(x)$ where $I_i(x)$ is the set of infinite directions defined in \eqref{appendix:infinite_direction_sets} and $\mathcal{I}_i(x)$ is defined in \eqref{infinite_directions_Si}. As a consequence the set of finite and infinite directions defined in \eqref{finite_directions_S} and \eqref{infinite_directions_S}, respectively, coincide with the sets of directions defined in this Appendix, that is, $\mathcal{F}(x)=F_g(x)$ and $\mathcal{I}(x)=I_g(x)$. The radial functions $\rho_i^\epsilon (x,v)$ are equal to the radial functions defined in  \eqref{appendix:radial_function_i}, and so, $\rho_\epsilon(x,v)$ is equal to the radial function defined in  \eqref{appendix:radial_function}. Now, let us prove that the family of functions $g_i's$ that we defined above satisfy the $\eta_\theta$-growth condition given in Definition \ref{def:thetagrowth:Appendix} for all $v\in \mathcal{I}(\bar{x})$. Consider $\bar{v}\in \mathcal{I}(\bar{x})$.  Since, the family of set-valued mappings $\S_i$ satisfies the $\eta$-growth condition for set-valued mappings at $\bar x$, we have that there exits $\hat{l}>0$ such that the family $\S_i$ satisfies \eqref{thetagrowth2:setvalued}. Let us set $l:=\hat{l}\epsilon$ and consider  $(x,v) \in \mathbb{B}_{1/l} (\bar{x}) \times
		\mathbb{B}_{1/l} (\bar{v})$ with $v\in \mathcal{F}(x)$ and $ \rho^i_\epsilon(x,v) \geq l$ with $i\in T_x^\epsilon(v)$. Let us notice that $\| \zep{\epsilon}{x}{v} - P_{\S_i(x)}(\zep{\epsilon}{x}{v})\|=\epsilon$, where $\zep{\epsilon}{x}{v}$ is defined in \eqref{definitionzep}. Now, by Lemma \ref{lemma:derivative}, the fact that $\eta$ is nondecreasing, $P_{\S_i(x)}(0)=0$ and by the non-expansiveness of the projection mapping we obtain the following sequence of inequalities
		\begin{align*}
			\| \pi_x(\pmord_x g_i(x,\zep{\epsilon}{x}{v}))\|  &
			\leq \|  D^\ast \S_i(x,P_{\S_i(x)}(\zep{\epsilon}{x}{v}))\|\epsilon
			\leq \hat{l}\epsilon \eta(\|P_{\S_i(x)}(\zep{\epsilon}{x}{v})\|)
			\leq  l \eta(\|\zep{\epsilon}{x}{v}\|).
		\end{align*}
		Hence it is enough to consider $\eta_{\theta}(\rho_\epsilon(x,v),v)=\eta(\|\zep{\epsilon}{x}{v}\|)$. Therefore, using Theorem \ref{thm:clarkeappli:Appendix} and the fact that 
		\begin{align*}
			\mathcal{G}(x,v)=\mathcal{M}_\epsilon(x,v)
		\end{align*}
		for all $v\in \mathbb{S}^{ m-1}$, we conclude that $\varphi_\epsilon$ is locally Lipschitz around $\bar{x}$ and that
		\begin{align*}
			\pmord \varphi_\epsilon (\bar{x})\subseteq -\cl{{w}^\ast }\left(  \int\limits_{\mathbb{S}^{m-1}}  	 \mathcal{M}_\epsilon(\bar{x},v)   d\mu_\zeta(v)\right)
		\end{align*}
	Finally, upon noticing that the family of set-valued mappings $\S_i$ satisfies the $\eta$-growth condition for set-valued mappings at all $x\in U':=\mathbb{B}_{\frac{1}{2\hat{l}}}(\bar{x})$ or more specifically the family satisfies \eqref{thetagrowth2:setvalued} with $3\hat{l}$, we obtain \eqref{eq:pmord_inclusion} where without loss of generality we considered $\hat{l}$ large enough such that $U'\subset U$.
	 
	\bibliographystyle{plain}
	\bibliography{biblio_vanAckooij[up2date@15102015],references}
\end{document}